\documentclass[11pt]{article}
\usepackage[utf8x]{inputenc}
\usepackage{microtype}
\usepackage{comment}
\usepackage[T1]{fontenc}
\usepackage{ae,aecompl}
\usepackage{times}
\usepackage{float}
\usepackage{verbatim,amsmath,amsthm,amsfonts,amssymb,latexsym,graphicx,mathtools,extpfeil,color,mathabx}
\usepackage{stmaryrd}
\usepackage{epstopdf,pinlabel} 
\usepackage[all]{xy}
\usepackage{graphicx}
\usepackage{caption}
\usepackage{subcaption}
\usepackage{slashed}
\DeclareMathAlphabet{\mathpzc}{OT1}{pzc}{m}{it}
\usepackage[colorlinks,pagebackref,hypertexnames=false]{hyperref} 
\usepackage{hyperref}
\hypersetup{
  colorlinks = true,
  linkcolor  = black
} 
\usepackage[alphabetic,backrefs,msc-links]{amsrefs}
\usepackage{accents}

\newcommand\CS{{\rm CS}}

\newcommand\SOt{\mathcal{SO}}
\newcommand\sutwo{\mathfrak{su}(2)}
\newcommand\B{\mathcal{B}}
\newcommand\A{\mathcal{A}}
\newcommand\G{\mathcal{G}}
\newcommand\Gzero{\G_0}
\newcommand\Bzero{\B_0}
\newcommand\energy{\mathcal{E}}
\newcommand\gr{\rm gr}

\usepackage{aliascnt}
\numberwithin{equation}{section}

\newcommand{\bQ}{{\bf Q}}
\newcommand{\bR}{{\bf R}}

\newcommand{\bZ}{{\bf Z}}

\newcommand{\fd}{{\mathfrak d}}

\newcommand{\fC}{{\mathfrak C}}
\newcommand{\fD}{{\mathfrak D}}

\newcommand{\fU}{{\mathfrak U}}

\newcommand{\Z}{\bZ}
\newcommand{\Q}{\bQ}
\newcommand{\R}{\bR}

\newcommand{\U}{{\rm U}}

\DeclareMathOperator{\Ad}{Ad}

\DeclareMathOperator{\Hom}{Hom}

\DeclareMathOperator{\tr}{tr}

\renewcommand{\epsilon}{\varepsilon}

\def\({\mathopen{}\left(}
\def\){\right)\mathclose{}}
\def\<{\mathopen{}\left<}
\def\>{\right>\mathclose{}}

\usepackage{multicol, color}

\definecolor{gold}{rgb}{0.85,.66,0}
\definecolor{cherry}{rgb}{0.9,.1,.2}
\definecolor{burgundy}{rgb}{0.8,.2,.2}
\definecolor{orangered}{rgb}{0.85,.3,0}
\definecolor{orange}{rgb}{0.85,.4,0}
\definecolor{olive}{rgb}{.45,.4,0}
\definecolor{lime}{rgb}{.6,.9,0}
\definecolor{green}{rgb}{.2,.7,0}
\definecolor{grey}{rgb}{.4,.4,.2}
\definecolor{brown}{rgb}{.4,.3,.1}

\def\makeautorefname#1#2{\AtBeginDocument{\expandafter\def\csname#1autorefname\endcsname{#2}}}

\newcommand{\mynewtheorem}[2]{
  \newaliascnt{#1}{equation}          
  \newtheorem{#1}[#1]{#2}
  \aliascntresetthe{#1}
  \makeautorefname{#1}{#2}
}
\mynewtheorem{theorem}{Theorem}
\mynewtheorem{prop}{Proposition}
\mynewtheorem{cor}{Corollary}
\mynewtheorem{construction}{Construction}
\mynewtheorem{lemma}{Lemma}
\mynewtheorem{conjecture}{Conjecture}
\mynewtheorem{question}{Question}
\mynewtheorem{problem}{Problem}

\numberwithin{substep}{step}
\makeautorefname{step}{Step}
\makeautorefname{substep}{Step}

\numberwithin{subcase}{case}
\makeautorefname{case}{Case}
\makeautorefname{subcase}{case}

\theoremstyle{remark}
\mynewtheorem{remark}{Remark}

\theoremstyle{definition}
\mynewtheorem{definition}{Definition}
\mynewtheorem{example}{Example}
\mynewtheorem{exercise}{Exercise}
\mynewtheorem{convention}{Convention}
\newtheorem*{convention*}{Convention}
\newtheorem*{conventions*}{Conventions}
        
\makeautorefname{chapter}{Chapter}
\makeautorefname{section}{Section}
\makeautorefname{subsection}{Section}
\makeautorefname{subsubsection}{Section}

\newcommand{\Addresses}{{
}}

\theoremstyle{plain}
\newtheorem{theorem-intro}{Theorem}
\newtheorem{conjecture-intro}[theorem-intro]{Conjecture}
\newtheorem{cor-intro}[theorem-intro]{Corollary}
\newtheorem{prop-intro}{Proposition}

\title{\large \bf Handle decomposition complexity and representation spaces}
\begin{document}

\author{Paolo Aceto
 \hspace{1cm} Aliakbar Daemi
\hspace{.75cm} Jennifer Hom 
\hspace{1cm}\\[5mm]
Tye Lidman
\hspace{1cm} JungHwan Park
}
\date{}
\maketitle
\begin{abstract}
We prove that there are homology three-spheres that bound definite four-manifolds, but any such bounding four-manifold must be built out of many handles.  The argument uses the homology cobordism invariant $\Gamma$ from instanton Floer homology.
\end{abstract}
\maketitle
\section{Introduction}
It is well known that every closed orientable three-manifold bounds a smooth, compact, orientable four-manifold \cite{Lickorish, Wallace}.  Consequently, it is natural to understand the minimal complexity of a bounding four-manifold.  One way to measure this is through the minimal number of handles in a handle decomposition.  If the first Betti number of the three-manifold is large, then it is easy to see that any handle decomposition of a bounding four-manifold is necessarily comparably large.  On the other hand, if the homology of a three-manifold is small, it can be difficult to show complexity in the handle decomposition of a bounding four-manifold, even if one expects this for three-manifolds which are complicated in other measures.  In~\cite{AGL}, it is shown that for certain rational homology spheres with finite non-trivial homology, all bounding rational homology balls must have many handles.  (This theorem is not vacuously true as their examples do bound rational balls.)

The goal of this paper is to exhibit examples of integer homology spheres requiring large handle decompositions of bounding four-manifolds with definite intersection forms, e.g.\ rational homology balls.  In this paper, we take the convention that all manifolds are smooth and a handle decomposition of a four-manifold with connected boundary $Y$ is a cobordism from the empty set to $Y$. Without loss of generality, we also focus on handle decompositions with one 0-handle and no 4-handles. An elementary but useful observation for us is that $\pi_1(W)$ is governed by the 1- and 2-handles. 
  
\begin{theorem-intro}\label{thm:main}
Let $Y_n = \mathbin{\#}_n \Sigma(2,3,5)$.  Then, in any handle decomposition of a definite four-manifold $W$ with boundary $Y_n \mathbin{\#} -Y_n$ the number of 1-handles is at least $n + b_1(W)$, the number of 2-handles is at least $n$, and the combined number of 2- and 3-handles is at least $n + b_1(W)$. 
\end{theorem-intro}

Note that $Y_n \mathbin{\#} - Y_n$ bounds an integer homology ball and hence Theorem \ref{thm:main} is not vacuously true. For instance, we may take $W$ to be the integer homology ball $I \times (Y_n \setminus B^3)$. Since $Y_n$ has Heegaard genus $2n$, this homology ball admits a handle decomposition where the numbers of 1- and 2-handles are both equal to $2n$, and there are no 3-handles. Theorem \ref{thm:main}, in particular, asserts that for any handle decomposition of any rational homology ball bounding $Y_n \mathbin{\#} - Y_n$ the numbers of 1- and 2-handles are  at least $n$. The method of the proof of Theorem \ref{thm:main} can be used to obtain lower bounds for the number of 1- and 2-handles of rational homology balls bounding some other families of homology spheres without giving any constraints on the number of 3-handles. (See the discussion at the end of the introduction.)  Note that it seems to still be open if every homology sphere bounding a homology ball bounds one without 3-handles (see \cite{DLVVW}).

We make the following conjecture which would generalize Theorem \ref{thm:main}. 
\begin{conjecture-intro}\label{con:main}
	 The minimum number of handles in any bounding four-manifold for $Y_n \mathbin{\#} - Y_n$ goes to infinity as $n$ goes to infinity.
\end{conjecture-intro}

\begin{remark}
By the work of Gordon-Luecke, $Y_n \mathbin{\#} -Y_n$ cannot bound a four-manifold built with no 1-handles and a single 2-handle \cite{GL}.  Taubes's periodic ends theorem implies $Y_n \mathbin{\#} -Y_n$ cannot bound a four-manifold with one 1-handle and one 2-handle \cite{Taubes}.
\end{remark}

The proof of Theorem~\ref{thm:main} relies on the following key technical result.
\begin{theorem-intro}\label{thm:extension}
Let $Y_n$ be as above.  Suppose $W\colon Y_n \to Y_n$ is a definite cobordism.  Then every representation in some $3n$-dimensional component of the $SU(2)$-representation variety of $\pi_1(Y_n)$ extends over $W$.  
\end{theorem-intro}
We can quickly deduce Theorem~\ref{thm:main} from Theorem~\ref{thm:extension}.  For notation, given a path-connected space $X$, let $R(X) = \Hom(\pi_1(X), SU(2))$.  If $\pi_1(X)$ is finitely-generated, e.g.\ if $X$ is a compact manifold, then $R(X)$ is a real algebraic variety.  For instance, if $\pi_1(X)$ is the free group $F_k$ on $k$ generators, then $R(F_k) = SU(2)^{\times k}$, which is $3k$-dimensional. Furthermore, $R$ is contravariant: a continuous map $\iota\colon X \to Y$ induces a morphism of real algebraic varieties $\iota^*\colon R(Y) \to R(X)$; if $\iota$ is $\pi_1$-surjective, then $\iota^*$ is an embedding.  Consequently, if $\pi_1(X)$ can be generated with $k$ elements, then the dimension of $R(X)$ as a real algebraic variety is at most $3k$.  For the reader who prefers manifolds, generic points in a real algebraic variety are manifold points, and the dimension is just the dimension of the tangent space at such a point.  One example for the reader to keep in mind is the union of the coordinate axes in $\mathbb{R}^2$.  For the relevant background on real algebraic geometry, the reader is referred to \cite{BCR}.    

\begin{proof}[Proof of Theorem~\ref{thm:main}]
Let $Q$ be a definite four-manifold bounding $Y_n \mathbin{\#} -Y_n$ and choose a handle decomposition for $Q$.
We first claim that if $b_1(Q)=0$, then the number of 2-handles is at least $n$.
For any such $Q$, attach a 3-handle to $Q$ to obtain a definite cobordism $W$ from $Y_n$ to $Y_n$. Let $\iota$ denote the inclusion of the incoming end.  Theorem~\ref{thm:extension} asserts that a $3n$-dimensional component of $R(Y_n)$ is contained in the image of $\iota^* \colon R(W) \to R(Y_n)$.  Since $\iota^*$ is a morphism of real algebraic varieties, we see that $R(W)$ has dimension at least $3n$.  Since attaching 3-handles does not affect $\pi_1$, we see that $R(Q)$ has dimension at least $3n$.  If $Q$ has a handle decomposition with $k$ 1-handles, then $R(Q)$ embeds in $R(F_k) = SU(2)^{\times_k}$, which is $3k$-dimensional.  Hence, $k\geq n$. Because $b_1(Q) = 0$, the number of $2$-handles has to be at least $n$.  This completes the proof of the claim.

Now, suppose that $b_1(Q) > 0$ and the number of 1-handles in a handle decomposition is $b_1(Q) + j$. By surgery along $b_1(Q)$ curves giving a basis for $H_1(Q;\R)$, we can obtain a new four-manifold $X$ with $b_1(X) = 0$. In fact, this surgery can be done in a compatible with the handle decomposition in a way that we trade $b_1(Q)$ 1-handles with $b_1(Q)$ 2-handles. One way to see this is to view the handle decomposition as a Kirby diagram, where the 1-handles are dotted circles.  Pick an appropriate set of $b_1(Q)$ of the 1-handles in $Q$  and replace those dotted circles by 0-framed 2-handles with attaching sphere given by those circles.  The resulting manifold $X$ now has a handle decomposition with $j$ 1-handles.  Note that $X$ is definite since we assumed $Q$ was.  Therefore, $j \geq n$. 

Since $b_1(Q)$ is determined by the 1- and 2-handles of $Q$ and there are at least $b_1(Q)+n$ 1-handles, it follows that the number of 2-handles must be at least $n$. That there are at least $b_1(Q) + n$ handles of index 2 or 3 follows from the fact that attaching 2- and 3-handles induces a cobordism from $\mathbin{\#}_{b_1(Q) + n} S^2 \times S^1$ to $Y_n \mathbin{\#} -Y_n$, which is a homology sphere.  Therefore, at least $b_1(Q) + n$ handles are necessary.
\end{proof}

The results can easily be extended to a hyperbolic example.  
\begin{cor-intro}
There exists a hyperbolic homology sphere $Z_n$ which bounds a homology ball and in any handle decomposition of any definite four-manifold $X$ with boundary $Z_n$ the number of 1-handles is at least $n + b_1(X)$ and the number of 2-handles is at least $n$.  Moreover, the combined number of 2- and 3-handles is at least $n + b_1(X)$. 
\end{cor-intro}

\begin{proof}
First, consider a homology cobordism $W_n$ with incoming end $Y_n \mathbin{\#} -Y_n$ obtained by attaching a single 1-handle to $Y_n \mathbin{\#} - Y_n$ and then a 2-handle along a hyperbolic knot in $Y_n \mathbin{\#} - Y_n \mathbin{\#} S^2 \times S^1$ which generates homology.  Such a hyperbolic knot exists by work of Myers \cite[Theorem 1.1]{Myers}.  By Thurston's hyperbolic Dehn surgery theorem, for all but finitely many framings, the outgoing end will be a hyperbolic homology sphere.  Choose such a generic framing and denote the resulting homology sphere by $Z_n$.  Since $Y_n \mathbin{\#} -Y_n$ bounds a  definite four-manifold, so does $Z_n$.  If $Z_n$ bounded a definite four-manifold $X$ with $j$ 1-handles, we could add $-W_n$ to obtain a definite four manifold with boundary $Y_n \mathbin{\#} -Y_n$.  Since $-W_n: Z_n \to Y_n \mathbin{\#} - Y_n$ has a handle decomposition with one 2-handle and one 3-handle, $X \cup -W_n$ gives a definite manifold with boundary $Y_n \mathbin{\#} -Y_n$ and a handle decomposition with $j$ 1-handles.  Theorem~\ref{thm:main} implies that $j \geq n+b_1(X)$.  
The statement about the number of 2-handles follows as in the proof of Theorem~\ref{thm:main}. The statement about the combined number of 2- and 3-handles also follows as in the proof of Theorem~\ref{thm:main}, where we consider the cobordism from $\mathbin{\#}_{j} S^2 \times S^1$ to $Z_n$, which is a homology sphere. 
\end{proof}

We now give a rough sketch of the proof of Theorem~\ref{thm:extension}.  Instanton Floer homology can be used to associate a homology cobordism invariant $\Gamma_Y \colon \Z \to \R_{\geq 0} \cup \{\infty\}$ to any homology sphere $Y$ \cite{Da:CS-HCG}.  If this function is suitably non-trivial, then it is shown that for any negative-definite cobordism $W$ from $Y$ to itself, there is a non-trivial $SU(2)$ representation on $Y$ that extends over $W$.  In some cases, we can gain more control over which representations extend.  In particular, we establish the following technical result, which may be of interest to experts.   For notation, let $\chi(X) = R(X)/conj$ for any topological space $X$.
\begin{theorem-intro}\label{thm:big-extension}
	Suppose the Chern--Simons functional on an integer homology sphere $Y$ is Morse--Bott and $0<\Gamma_Y(i)<\infty$. Let $W\colon Y \to Y$ be a definite cobordism.  Then there are connected components $\chi_0$, $\chi_1$ of $\chi(Y)$ such that 
	\begin{itemize}
		\item[(i)] $\CS(\chi_0)=\CS(\chi_1)=\Gamma_Y(i)$;
		\item[(ii)] for any $\rho_0\in \chi_0$ and $\rho_1\in \chi_1$, there is $\sigma\in \chi(W)$ that restricts respectively to $\rho_0$ and $\rho_1$ 
		on the incoming and the outgoing ends of $W$;
		\item[(iii)] ${\gr}(\chi_0) + \dim(\chi_0) \geq 4i - 3$ and ${\gr}(\chi_1) + \dim(\chi_1) \geq 4i - 3$.
	\end{itemize}
\end{theorem-intro}

Here, $\CS$ denotes the Chern--Simons invariant of any element of $\chi_0$. The Chern--Simons invariant is defined for any element of $\chi(Y)$ (or more generally any $SU(2)$ connection on $Y$) and is constant on each connected component of $\chi(Y)$. Furthermore, ${\gr}(\chi_0)$ denotes the integer valued Floer grading of any element of $\chi_0$. (See Section \ref{sec:background} for a review of the definition of these two quantities.) A priori, $\CS(\chi_0)$ (resp. ${\gr}(\chi_0)$) is well-defined up to an integer  (resp. an integer multiple of $8$). However, ${\gr}(\chi_0)-8\CS(\chi_0)$ is a well-defined real number. The claim in Theorem \ref{thm:big-extension} about $\chi_0$ should be interpreted in the following way: the Chern--Simons invariant of $\CS(\chi_0)$ is equal to $\Gamma_Y(i)$ mod integers and the well-defined real number ${\gr}(\chi_0)-8\CS(\chi_0)+ \dim(\chi_0)$ is at least $4i-8\Gamma_Y(i) - 3$.

The Morse-Bott condition in Theorem \ref{thm:big-extension} is satisfied for any Seifert fibered homology sphere and is also closed under connected sums.  Therefore, we may apply the theorem to $Y_n$. In Proposition~\ref{prop:Gamma-Yn} below, we show that $\Gamma_{Y_n}(2n) = \frac{49n}{120}$, and so there is some component $\chi_0$ of $\chi(Y)$ with Chern--Simons invariant $\frac{49n}{120}$ for which every representation extends over $W$.  By comparing the Chern--Simons invariants of representations on $Y_n$ and grading information in the instanton Floer homology of $Y_n$, we are able to show that the corresponding component of $R(Y)$ has dimension $3n$.  

\begin{remark}\label{rmk:Yn-extension-choice}
For $Y_n$, we can identify a concrete choice of $\chi_0$ and $\chi_1$ precisely satisfying the conditions of the theorem (and they are the same).  See Example~\ref{ex:phs}, Example~\ref{ex:chi-Yn}, and Corollary~\ref{cor:beta-extends} below.  In fact, these collected results show that there is a choice of $\chi_0 = \chi_1$ on $Y_n$ which can be taken independent of the choice of $W$ and $\chi_0$ is preserved under any self-diffeomorphism of $Y_n$.  This $\chi_0$ is a $3n-3$-dimensional component of $\chi(Y_n)$.  
\end{remark}

Theorem \ref{thm:big-extension} also has a natural application to homology $S^3 \times S^1$'s and 2-knots.  This was suggested to us by Masaki Taniguchi.  
\begin{cor-intro}
Let $X$ be a closed four-manifold with the same integral homology as $S^3 \times S^1$ and suppose $Y_n$ is a non-separating hypersurface.  Then, there is a $3n$-dimensional component of $R(Y_n)$ which extends over $R(X)$, and hence $\pi_1(X)$ requires at least $n$ generators.    
\end{cor-intro}
\begin{proof}
Cut $X$ along $Y_n$ to obtain a homology cobordism $W\colon Y_n \to Y_n$.  By Theorem~\ref{thm:big-extension} and Remark~\ref{rmk:Yn-extension-choice} there is a $3n-3$-dimensional component of $\chi_0$ of $ \chi(Y_n)$ such that for each $\rho \in \chi_0$ and $\rho' \in \chi_0$, there is $\rho_W \in \chi(W)$ such that $\rho_W$ restricts to $\rho$ and $\rho'$ on the ends.  Now, for each $\rho \in \chi_0$, choose the $\rho'$ such that $\rho$ and $\rho'$ are identified by the gluing map used to construct $X$ from $W$.  (Here we are using that a self-diffeomorphism of $Y_n$ preserves $\chi_0$.)  The extension $\rho_W$ can then be extended over $X$, which completes the proof.  For the benefit of the reader, we spell out the gluing more explicitly in the language of representations, since a priori, it seems as though there is a conjugacy indeterminacy.  (This can also easily be understood by looking at flat connections where one glues up the bundles on the ends by a suitable automorphism.)  Letting $\iota^\pm$ denote the inclusions of $Y_n$ into $W$ as the incoming and outgoing ends of $W$, we have a description of $\pi_1(X)$ as an HNN extension as follows: 
\[
\pi_1(X) = \langle \pi_1(W) , t \mid \iota^-_*(y) = t^{-1}  \iota^+_*(y) t, \ y \in \pi_1(Y_n) \rangle.
\] 
Take $R_0$ to be the corresponding $3n$-dimensional component of $R(Y_n)$ coming from $\chi_0$.  For each $\rho \in R_0$, we can find a corresponding $\rho' \in R_0$ given by the gluing of the two ends.  (By convention, we take $\rho$ to be on the incoming end.)  Theorem~\ref{thm:big-extension} guarantees that there is $\rho_W$ on $W$ such that $\rho_W$ restricts to $\rho$ on the incoming end and a conjugate of $\rho'$ on the outgoing end.  Suppose that the conjugacy is realized by an element $A \in SU(2)$.  Then, we define a representation of $\pi_1(X)$ by $\rho_W$ on $\pi_1(W)$ and sending $t$ to $A$.  Hence, we can extend the $3n$-dimensional component of $R(Y_n)$ over $X$.
\end{proof}

Since a 2-knot in $S^4$ has the same knot group as surgery on that 2-knot, we have the following.  
\begin{cor-intro}
Let $K$ be a 2-knot in $S^4$ for which a once-punctured $Y_n$ is a Seifert solid.  Then, $\pi_1(S^4 - K)$ requires at least $n$ generators.  
\end{cor-intro}

\subsection*{More general three-manifolds}

The three-manifold $Y_n \mathbin{\#} - Y_n$ is the branched double cover of the knot $K_n\mathbin{\#}-K_n$ where $K_n = \mathbin{\#}_n T_{3,5}$ and $-K_n$ is the reverse of the mirror image of $K_n$. For any embedded disc $F$ in the the $4$-ball $B^4$ bounding $K_n\mathbin{\#}-K_n$, we may take the double cover of $B^4$ branched over $F$ to obtain a rational homology ball $\Sigma(F)$. Moreover, any relative handle decomposition of $F$ induces a handle decomposition of $\Sigma(F)$ with as many handles as the handles building $F$. Thus, lower bounds on the number of handles in a bounding four-manifold for $Y_n \mathbin{\#} - Y_n$ give lower bounds on the fusion number of $K_n\mathbin{\#}-K_n$. In particular, the fusion number of the determinant $1$ knot $K_n\mathbin{\#}-K_n$ grows linearly in $n$. For this particular family of knots, one can look at the ordinary homology groups of other cyclic branched covers of $K_n\mathbin{\#}-K_n$ to obtain similar lower bounds for the fusion number of $K_n\mathbin{\#}-K_n$ in a simpler way. It is natural to ask whether the techniques in this paper can be used to obtain lower bounds on fusion numbers that are not accessible to other techniques.

Motivated by this, we first state a slightly weaker version of Theorem \ref{thm:main} for more general families of three-manifolds. 

\begin{theorem-intro} \label{thm-Gamma}
	Suppose $Y$ is an integer homology sphere such that the Chern--Simons functional is Morse-Bott and for a positive integer $i_0$, we have $\Gamma_Y(i_0)<i_0/2$. Then the minimum number of $1$-handles and $2$-handles among definite four-manifolds 
	that fill $\mathbin{\#}_nY\mathbin{\#}_n-Y$ grows without any upper bound as $n\to \infty$.
\end{theorem-intro}

Examples of three-manifolds satisfying the assumption on $\Gamma_Y$ in the above theorem are given in \cite{Da:CS-HCG}. For instance, any Seifert fibered homology sphere $\Sigma(a_1,\dots,a_n)$ with a positive value of Fintushel and Stern’s $R$-invariant from \cite{FS:pseudofree} satisfies the assumption in Theorem \ref{thm-Gamma} \cite[Theorem 6]{Da:CS-HCG}. Furthermore, we can take $Y$ to be the branched double cover of the Whitehead double $D(T_{p,q})$ of the $(p,q)$-torus knot. Then \cite[Corollary 3]{Da:CS-HCG} asserts that  $\Gamma_Y(1)$ is less than $\frac{1}{4pq(2pq-1)}$, and hence smaller than $1/2$. The Chern--Simons functional of these three-manifolds is also Morse-Bott; see \cite[Section 4]{FS:HFSF} for the case of Seifert fibered homology spheres and Proposition \ref{Dpq-MB} for the branched double cover of $D(T_{p,q})$. Consequently, we have the following corollary of Theorem \ref{thm-Gamma}.

\begin{cor-intro}\label{cor:alex}
There are Alexander polynomial one knots $K$ such that the fusion number of the ribbon knot $J_n=\mathbin{\#}_nK\mathbin{\#_n}-K$ grows without any upper bound as $n \to \infty$.
\end{cor-intro}

For instance, we can take $K=D(T_{p,q})$. The assumption on the Alexander polynomial implies that all cyclic branched covers of $J_n$ are integer homology spheres, and one cannot obtain constraints on the fusion number of $J_n$ from the homology groups of cyclic branched covers. Note that there are other modern invariants giving lower bounds on fusion number such as \cite{JMZ:tor}. In particular, letting $J_n$ be the $(n,1)$-cable of the Kinoshita-Terasaka knot results in a family of Alexander polynomial one knots $\{J_n\}$, where the fusion number of $J_n$ is $n$ \cite[Theorem 1.1]{HKP}. However, the torsion invariant of \cite{JMZ:tor} provides the same lower bound for each knot of the form $\{\mathbin{\#}_nK\mathbin{\#_n}-K\}$ for a fixed knot $K$, and so would not work for the knots $\mathbin{\#}_n D(T_{p,q}) \mathbin{\#_n} -D(T_{p,q})$.

Furthermore, we remark that Theorem~\ref{thm-Gamma} in fact implies that the minimum number of generators and relations of the fundamental group of definite four-manifolds that fill $\mathbin{\#}_nY \mathbin{\#_n} -Y$ grows without any upper bound as $n \to \infty$, where $Y$ is the branched double cover of $D(T_{p,q})$. In particular, it provides the first example of a family of Alexander polynomial one knots that have arbitrarily large homotopy fusion number. Recall that the \emph{homotopy fusion number} of a ribbon knot is defined as the minimal number of relations for a presentation of the fundamental group of a ribbon disk complement. By definition, it gives a lower bound for the fusion number. In contrast to our examples in Corollary~\ref{cor:alex}, the examples obtained by $(n,1)$-cabling a fixed knot have a bounded homotopy fusion number, determined by the fusion number of the companion knot~\cite[Proposition 2.1]{HKP}.

\subsection*{Organization}
The paper is organized as follows.  In Section~\ref{sec:background} we review the basics of Chern--Simons gauge theory and its connection with the representation variety.  In Section~\ref{sec:SOcomplexes} we review the necessary algebra and gauge theory to define an instanton invariant for a restricted class of homology spheres. The definition of $\Gamma$ for this class of homology spheres is recalled in Section~\ref{sec:gamma-phs}, where it is computed for connected sums of the Poincar\'e homology sphere.  Finally, in Section~\ref{sec:extensions} we prove Theorem~\ref{thm:big-extension} and use it to deduce Theorem~\ref{thm:extension}.

\subsection*{Acknowledgements.} 
We thank Chuck Livingston for helpful discussions, and Masaki Taniguchi for helpful discussions and for pointing out consequences related to 2-knots and the homotopy fusion number. PA was supported by the European Union’s Horizon 2020 research and innovation programme under the Marie Sk\l odowska-Curie action, Grant No.\ 101030083 LDTSing. AD was supported by NSF Grants DMS-1812033, DMS-2208181 and NSF FRG Grant DMS-1952762NSF.  JH was partially supported by NSF grant DMS-2104144 and a Simons Fellowship.  TL was supported by NSF grants DMS-17909702, DMS-2105469, and a Sloan Fellowship, and additionally thanks the department at Washington University in St. Louis for its hospitality.  JP was supported by Samsung Science and Technology Foundation (SSTF-BA2102-02) and the POSCO TJ Park Science Fellowship.  This work was completed while AD, JH, and TL were in residence at the Simons Laufer Mathematical Sciences Institute during Fall 2022, supported by NSF Grant DMS-1928930.

\section{Chern-Simons gauge theory and representations}\label{sec:background}

In this section, we review the Chern-Simons functional and how it relates to representations of three-manifold groups.  Throughout, let $Y$ be an integer homology three-sphere and let $P$ denote the trivial $SU(2)$ bundle.    

Let $\mathcal{A}(Y)$ denote the space of connections on $P$, which can be identified with $\Omega^1(Y; \sutwo)$ using the trivialization of $P$.  
We define the {\em Chern-Simons} functional as
\[
\CS \colon \mathcal{A}(Y) \to \R, \hspace{1cm} \CS(A) = \frac{1}{8\pi^2} \int_Y \tr(\alpha \wedge d \alpha+\frac{2}{3} \alpha \wedge \alpha  \wedge \alpha).
\]
The critical points of the Chern-Simons functional are precisely those satisfying $F_\alpha = 0$, i.e.\ the {\em flat connections}.  In particular, the value of the Chern-Simons is constant on the connected components of the critical set.

The Chern--Simons functional behaves well with respect to the action of the {\it gauge group} $\G(Y) = {\rm Map}(Y,SU(2))$.  Any element $g \in \G(Y)$ acts on a connection by $g^*\alpha = g^{-1}dg + g\alpha g^{-1}$ and 
\[\CS(g^*\alpha) = \CS(\alpha) + \deg(g),\] 
where $\deg(g)$ is the degree of the map between three-manifolds.  Let $\Gzero(Y)$ denote the subgroup of $\G(Y)$ consisting of degree zero gauge transformations. 
Therefore, the Chern-Simons functional descends to
\[
\CS\colon \B(Y) \to \R/\Z, \hspace{2cm} \CS \colon \Bzero(Y) \to \R.
\]
where $\B(Y) = \A(Y)/\G(Y)$ and $\Bzero(Y) = \A(Y) / \Gzero(Y)$. With a slight abuse of terminology, we call an element of $\B(Y)$ a connection and use the term {\it lifted connection} to refer to an element of $\Bzero(Y)$. A connection is {\em irreducible} if the stabilizer of $\alpha$ is precisely the set of constant gauge transformations $\{\pm 1\}$. In the following, we also write $\A^*(Y)$ for the irreducible connections in $\A(Y)$, and similarly for $\B^*(Y), \Bzero^*(Y)$.

Associated to a flat connection $A$, there is an associated holonomy representation $\pi_1(Y) \to SU(2)$, defined up to conjugation.  In fact, there is a homeomorphism 
\begin{equation}\label{flat-rep}
	Crit(\CS)/\G(Y) \cong R(Y)/conj,  
\end{equation}
where $R(Y)/conj$ denotes modding out $R(Y)$ by conjugation.  We write $\chi(Y) = R(Y)/conj$, the {\em character variety} of $Y$. 
Therefore, we can discuss the Chern-Simons invariant of an element of $R(Y)$ or $\chi(Y)$ as an element of $\R/\Z$, and it is constant on connected components. Since $Y$ is an integer homology sphere, the equivalence class of the trivial connection $\theta$ (or the trivial representation using the correspondence \eqref{flat-rep}) gives the only element of \eqref{flat-rep} that is reducible. To match our gauge theory notation, let $R^*(Y)$ and $\chi^*(Y)$ denote the result of removing the trivial representation. In particular, $R^*(Y)$ has a copy of $SO(3)$ for each element of $\chi^*(Y)$.
%

We may define a {\it relative grading} between any pair of connections on $Y$. For a connection $\alpha$ on $Y$, let 
\[
K_\alpha := \left[\begin{array}{cc} 0 & -d^*_\alpha \\ -d_\alpha  & * d_\alpha  \end{array}\right] \colon L^2_1(\Omega^0 \oplus \Omega^1) \to L^2(\Omega^0 \oplus \Omega^1),  
\]
where $\Omega^i = \Omega^i(Y, \sutwo)$.  The differential operator $K_\alpha$ is self-adjoint and Fredholm.  For any two connections $\alpha, \alpha'$ on $Y$ such that $K_\alpha$, $K_{\alpha'}$ are invertible, choose a smooth path $\{\alpha(t)\}_{-1\leq t \leq 1}$ with $\alpha(-1)=\alpha$, $\alpha(1)=\alpha'$, and consider the path of operators $K_{\alpha(t)}$.  We define $\gr(\alpha, \alpha')$ to be the spectral flow of this family. To extend this definition to arbitrary elements $\alpha$, $\alpha'$ of $\mathcal A(Y)$, we replace the operator $K_{\alpha(t)}$ with $K_{\alpha(t)}-\epsilon t\cdot {\rm Id}$ using a positive real number $\epsilon$ that is less than the magnitude of the non-zero eigenvalues of $K_{\alpha}$ and $K_{\alpha'}$. We call ${\gr}(\alpha, \theta)$ the {\em Floer grading}.  This spectral flow is invariant with respect to the action of $\G_0(Y)$. In fact, we have
\[
  {\gr}(g^*\alpha, h^*\alpha') = {\gr}(\alpha, \alpha')+8(\deg(g)-\deg(h))
\]
for any $g, h \in \G(Y)$. In summary, we obtain a $\Z/8$-valued Floer grading on the elements of $\B(Y)$ and a $\Z$-valued grading on the configuration space of lifted connections $\B_0(Y)$. Moreover, the quantity $\gr-8\CS$ gives a well-defined real-valued function on $\B(Y)$.

\begin{example}\label{ex:phs}
Let $Y_1 = \Sigma(2,3,5)$, oriented as the boundary of a negative-definite plumbing.  Then, $\chi(Y_1)$ consists of three isolated points: the trivial representation $\theta$ and two irreducible representations $\alpha, \beta$.  Hence, $R(Y)$ consists of a point and two copies of $SO(3)$.  For appropriate choices of lifts of $\alpha$ and $\beta$, we have 
\begin{equation}\label{bi-grading}
  (\CS(\alpha),{\gr}(\alpha))= \left(\frac{1}{120},1\right),\hspace{1cm}\left(\CS(\beta),{\gr}(\beta)\right)= \left(\frac{49}{120},5\right).
\end{equation}
See \cite[Theorem 3.7]{FS:HFSF} or \cite[Theorem 5.2]{KirkKlassen} for more details.
\end{example}

\begin{example}\label{connected-sum-inv}
Let $Y$ be a connected sum of two homology spheres $Y_1$, $Y_2$.  Then, $R(Y) = R(Y_1) \times R(Y_2)$.  The character variety $\chi(Y)$ consists of the trivial representation, a copy of $\chi^*(Y_1)$ (coming from $R^*(Y_1) \times \{\theta_2\}$ where $\theta_2$ is the trivial representation on $Y_2$), a copy of $\chi^*(Y_2)$, and $R^*(Y_1) \times R^*(Y_2) / conj$ where the conjugation action acts diagonally on $R^*(Y_1) \times R^*(Y_2)$.  Additionally, writing $\rho \in R(Y)$ as $(\rho_1, \rho_2)$, it is well known that 
\begin{equation}\label{additivity}
  \CS(\rho) = \CS(\rho_1) + \CS(\rho_2), \hspace{1cm} {\gr}(\rho)={\gr}(\rho_1)+{\gr}(\rho_2).
\end{equation}  
For instance, one can verify these identities using the fact that the disconnected three-manifold $Y_1\sqcup Y_2$, together with flat connections $\alpha_1$ and $\alpha_2$, is {\it flat cobordant} to $(Y,\alpha)$ via the 1-handle cobordism. One can then apply the methods of \cite{Auckly,Auckly:top-spec-flow} to relate the gauge theoretical invariants of $\alpha$, $\alpha_1$, and $\alpha_2$. (The claim about the values of Chern--Simons invariants is in fact more elementary, and one can see it directly from the definition.) To be more precise, the identities in \eqref{additivity} hold respectively mod $\Z$ and $8\Z$. However, one can similarly get the following additivity without any ambiguity:
\begin{equation}\label{additivity-2}
  {\gr}(\rho)-8\CS(\rho) = ({\gr}(\rho_1)-8\CS(\rho_1)) + ({\gr}(\rho_2)-8\CS(\rho_2)).
\end{equation}  
\end{example}

\begin{example}\label{ex:chi-Yn}
	Let $Y_n = \mathbin{\#}_n \Sigma(2,3,5)$, where each $\Sigma(2,3,5)$ is oriented as the boundary of a negative-definite plumbing.
	Given an $n$-tuple $\vec{\sigma}$ of the elements of $\{\theta, \alpha, \beta\}$, we write $\chi(\vec{\sigma})\subset \chi(Y_n)$ 
	for the connected component of the character variety consisting of elements which are conjugate to $\sigma_i$ on the $i$-th 
	$\Sigma(2,3,5)$-summand of $Y_n$. If $j$ and $k$ respectively denote the number of the copies of $\alpha$ and $\beta$ in $\vec{\sigma}$, and 
	$i=j+k$, then $\vec{\sigma}$ is diffeomorphic to the products of $(i-1)$ copies of $SO(3)$, 
	and the value of the Chern--Simons functional and ${\gr}$ on $\chi(\vec{\sigma})$ are respectively
	equal to $(j+49k)/120$ and $j+5k$ . In particular, there are $2^i \cdot {n \choose i}$ many connected components
	of $\chi(Y_n)$ of form $SO(3)^{\times_{(i-1)}}$. 
%
\end{example}
   

\section{Instanton Floer homology and $\SOt$-complexes}\label{sec:SOcomplexes}

Given any integer homology sphere $Y$, a function $\Gamma_Y \colon \Z\to \R^{\geq 0}\cup \infty$ is defined in \cite{Da:CS-HCG}, which is a homology cobordism invariant.  That is to say, if there is an integer homology cobordism $W \colon Y\to Y'$, then $\Gamma_Y=\Gamma_{Y'}$. 
This invariant is constructed using the package of instanton Floer homology, and the proof of Theorem~\ref{thm:main} relies on computing $\Gamma_{Y_n}$ for the connected sums of Poincar\'e homology spheres $Y_n$. This computation is carried out in Section~\ref{sec:gamma-phs}.  The goal of this section is to review the algebraic preliminaries and the necessary background on instanton Floer homology to define the $\Gamma$-invariant.  
The exposition here does not follow the original construction in \cite{Da:CS-HCG}, but instead follows \cite{DS:equiv-asp-sing} where analogous invariants for knots are constructed.  The main difference with \cite{Da:CS-HCG} is that here we work over $\Q[x^{\pm 1}]$ instead of a Novikov ring.  


To define the $\Gamma$-invariant, one needs an algebraic construction analogous to the notion of {\it enriched $\mathcal{S}$-complexes} of \cite{DS:equiv-asp-sing}, which one might call an {\it enriched $\mathcal{SO}$-complex} because the Lie group $S^1$ in \cite{DS:equiv-asp-sing} is replaced with $SO(3)$ in the present setup.  (See also \cite[Section 6.6]{DLVVW} for a similar setup.)
This algebraic object is somewhat complicated, but when the Chern--Simons functional has non-degenerate critical points and the moduli spaces of downward gradient flowlines are cut out transversely, one can use a simpler object called an {\em $I$-graded $\SOt$-complex}, and the $\Gamma$-invariant is easier to define.  The Poincar\'e homology sphere is an example of such a manifold.  There is also a connected sum theorem (Theorem~\ref{thm:kunneth} below), which implies that the enriched $\mathcal{SO}$-complex of a connected sum of Poincar\'e homology spheres can be represented by an $I$-graded $\SOt$-complex.  For this reason, we do {\em not} define the instanton enriched $\SOt$-complex or $\Gamma$-invariants for a general homology sphere.  To prove Theorem~\ref{thm:big-extension}, which applies to a broader family of homology spheres, we rely on one established property of $\Gamma$ (Lemma~\ref{lem:extension}), but do not need to use anything about the definitions. We hope this makes the article more palatable to the reader.

\subsection{$\SOt$-complexes}
In this subsection, we discuss an algebraic structure called an {\em $I$-graded $\SOt$-complex}. This consists of a module $\widetilde C$ over the ring $\Q[x^{\pm 1}]$ together with $\Q[x^{\pm 1}]$-module homomorphisms $\widetilde d \colon \widetilde C\to \widetilde C$ and $\chi \colon \widetilde C\to \widetilde C$. These homomorphisms are required to satisfy the following properties:
\begin{itemize}
	\item[(i)] $\widetilde d^2=0$, $\chi^2=0$ and $\widetilde d\chi+\chi\widetilde d=0$. That is to say, $\widetilde d$ and $\chi$ 
	are differentials that anti-commute with each other.
	\item[(ii)] There is a splitting of $\Q[x^{\pm 1}]$-modules 
	\begin{equation}\label{splitting}
		\widetilde C=C\oplus C\oplus \Q[x^{\pm 1}]
	\end{equation} 
	such that the homomorphism $\chi$ maps an element $(\alpha,\beta,r)\in C\oplus C \oplus \Q[x^{\pm 1}]$ to $(0,\alpha,0)$.  
	Here the summand $C$ is a finitely generated, free module over the ring $\Q[x^{\pm 1}]$. 
\end{itemize}
Finally $\widetilde C$ admits a $\Z$-grading and an $\R$-grading, respectively called the {\it homological grading} and {\it $I$-grading} and denoted by $\gr$ and $\deg_I$, such that $\widetilde d$ and $\chi$ are compatible with these gradings in the following sense. We first define a $(\Z\times \R)$-grading on $\Q[x^{\pm 1}]$ where $x^i$ has degree $(8i,i)$. For any $(k,r)\in \Z\times \R$, let $\Q[x^{\pm 1}]_{(k,r)}$ be the $(\Z\times \R)$-graded $\Q[x^{\pm 1}]$-module that is isomorphic to $\Q[x^{\pm 1}]$ as a module, but the $(\Z\times \R)$-grading is shifted up by $(k,r)$.  
We have the following requirements on the $(\Z\times \R)$-grading of $\widetilde C$.  
\begin{itemize}
	\item[(iii)] There is a finite list of pairs $(k_i,r_i)$ such that
		\begin{equation}\label{splitting-homog}
			C=\bigoplus_i \Q[x^{\pm 1}]_{(k_i,r_i)},
		\end{equation}
		and the graded version of \eqref{splitting} is given by 
		\begin{equation}\label{splitting-graded}
			\widetilde C=C\oplus C_{(3,0)}\oplus \Q[x^{\pm 1}]_{(0,0)}.
		\end{equation} 
		That is to say, the grading on the second summand is given by shifting the grading on the first summand by $(3,0)$. 
		In particular, the map $\chi$ increases the homological grading by $3$ and does not change the $I$-grading.
	\item[(iv)]	The map $\widetilde d$ has degree $-1$ with respect to the homological grading and decreases the $I$-grading.
\end{itemize}

The second part of (iv) needs some clarification. At this point, we defined the $I$-grading only for homogenous elements of \eqref{splitting-graded}. For an arbitrary $\zeta\in \widetilde C$, we can uniquely write 
\[
  \zeta=\zeta_1+\zeta_2+\dots+\zeta_n,
\]
where $\zeta_i$ belongs to a homogenous summand of \eqref{splitting-graded} induced by the splitting in \eqref{splitting-homog}. We define the $I$-grading of $\zeta$, denoted by $\deg_I(\zeta)$, as the maximum of the real values $\deg_{I}(\zeta_j)$. In particular, the $I$-grading of $0\in \widetilde C_i$ is $-\infty$. Condition (iv) implies that $\deg_I(\zeta)$ is greater than $\deg_I(\widetilde d \zeta)$ for any $\zeta$.  Even though $\deg_I$ is called a grading in the literature, it is in fact only a filtration.

Conditions (i) and (ii) imply that the differential $\widetilde d$ with respect to the splitting \eqref{splitting} has the matrix form
\begin{equation}\label{widetilde-d}
	\widetilde d=\left[\begin{array}{ccc}d&0&0\\U&-d&D_2\\D_1&0&0\end{array}\right],
\end{equation}
where $d\colon C\to C$, $U\colon C\to C$, $D_1 \colon C\to \Q[x^{\pm 1}]$ and $D_2\colon \Q[x^{\pm 1}]\to C$ are module homomorphisms. These maps satisfy the relations
\begin{equation}\label{so-comp-id}
  d^2=0,\hspace{1cm} D_1d=0,\hspace{1cm} dD_2=0,\hspace{1cm} dU-Ud=D_2D_1.
\end{equation}

\begin{example}
	Suppose $(\widetilde C,\widetilde d,\chi)$ and $(\widetilde C',\widetilde d',\chi')$ are two $I$-graded $\mathcal{SO}$-complexes. Then $\widetilde C^\otimes:=\widetilde C\otimes_{\Q[x^{\pm 1}]} \widetilde C'$ inherits a $(\Z\times \R)$-grading from $\widetilde C$ and 
	$\widetilde C'$. This bigraded module together with the homomorphisms 
	\[
	  \widetilde d^\otimes=\widetilde d\otimes 1+\epsilon\otimes \widetilde d',\hspace{1cm}\chi^\otimes=\chi\otimes 1+\epsilon\otimes \chi'
	\]
	defines an $I$-graded $\mathcal{SO}$-complex. Here $\epsilon:\widetilde C\to \widetilde C$ is the homomorphism that acts as $(-1)^i$ 
	on the elements with homological grading $i$.
\end{example}

There are also morphisms, chain homotopies, and chain homotopy equivalences of $I$-graded $\SOt$-complexes, which are defined similarly.  See for example \cite[Section 6.6]{DLVVW} or \cite[Section 4]{DS:equiv-asp-sing} for more details on similar constructions.  

\subsection{The $\SOt$-complex for instanton Floer homology}\label{subsec:SO-comp-Y}
Instanton Floer theory can be used to produce $I$-graded $\SOt$-complexes. We fix an integer homology sphere $Y$, and for the sake of exposition, we make the simplifying assumption that the Chern--Simons functional satisfies a variant of the Morse-Smale condition: we ask that the critical points of the Chern--Simons functional are non-degenerate and that the moduli spaces of downward gradient flowlines are cut out transversely. For any two flat connections $\alpha$ and $\beta$ on the trivial bundle over $Y$, we write $\breve M(\alpha,\beta)_d$ for the $d$-dimensional component of the moduli space of {\it unprametrized downward gradient flowlines of the Chern--Simons functional} from $\alpha$ to $\beta$.  If $\alpha$ and $\beta$ are {\em lifted} flat connections, we define $\breve M(\alpha,\beta)_d$ to be the same moduli spaces after projecting $\alpha$ and $\beta$ to $\B(Y)$.  (To be more precise, this moduli space needs to be defined in terms of the solutions of the ASD equation on $\R\times Y$ that are asymptotic to $\alpha$ and $\beta$ on the ends.) In particular, the moduli space $\breve M(\alpha,\beta)_d$ is empty only if $d$ mod $8$ is equal to ${\gr}(\alpha)-{\gr}(\beta)-1-h^0(\alpha)$, where $h^0(\alpha)$ is the dimension of the stabilizer of $\alpha$ with respect to the action of the gauge group.

In order to define the instanton $\SOt$-complex $\widetilde C(Y)$ of an integer homology sphere $Y$, we need to define $C(Y)$ and the maps $d$, $U$, $D_1$ and $D_2$ satisfying the properties mentioned in the previous subsection. We define $C(Y)$ to be the $\Q$-vector space spanned by the critical points of $\CS \colon \Bzero^*(Y) \to \R$. Applying an element of the gauge group with degree $1$ defines an invertible action $x$ on $\Bzero^*(Y)$, and this gives $C$ the structure of a free $\Q[x^{\pm 1}]$-module, which is finitely generated by flat connections on $Y$. We define a $(\Z\times \R)$-grading on $C(Y)$ using the Floer grading and the Chern--Simons functional. 

The operators $d, U, D_1$ and $D_2$ come from counting the elements of the moduli spaces $\breve M(\alpha,\beta)_d$.  Variants of these maps can be found in \cite{Fl:I,Don:YM-Floer,Fro:h-inv}, but we mention them briefly because we use slightly different conventions in the definitions of these maps. We essentially follow the conventions in \cite{Da:CS-HCG} except that we use the language of lifted flat connections instead of Novikov rings. We start with the operator $d$, and for any lifted flat connection $\alpha$ with ${\gr}(\alpha)=i$, we define  
\begin{equation}\label{Floer-diff}
 d(\alpha) = \sum_{\beta} \mathbin{\#} \breve M(\alpha, \beta)_0 \cdot \beta,
\end{equation}
where the sum is over all lifted flat connections $\beta$ with ${\gr}(\alpha)=i-1$. By extending linearly over $\Q$, we obtain the map $d\colon C(Y) \to C(Y)$, which is, in fact, linear over $\Q[x^{\pm 1}]$. For any lifted flat connection $\alpha$ in degree $8i+1$, we define 
\[
  D_1(\alpha) =\mathbin{\#} \breve M(\alpha, \theta)_0 \cdot x^i
\]
and for lifted flat connections in other degrees, this map vanishes. The operator $D_2$ is defined by the $\Q[x^{\pm 1}]$-linear extension of the map
\[
  D_2(1) = \sum_{\beta} \mathbin{\#} \breve M(\theta, \beta)_0\cdot  \beta,
\]
where the sum is over all $\beta$ that ${\gr}(\beta)=-4$. The operator $U:C(Y)\to C(Y)$, which decreases the Floer grading by $4$, is defined analogous to \eqref{Floer-diff} except that $\#\breve M(\alpha, \beta)_0$ is replaced with $c\cdot \#N(\alpha, \beta)_0$ where $c$ is a fixed constant and $N(\alpha, \beta)_0$ is a co-dimension three subspace of $\#\breve M(\alpha, \beta)_3$. In  \cite{Don:YM-Floer,Fro:h-inv,Da:CS-HCG} different choices for the constant $c$ is used, and we use the convention of \cite{Da:CS-HCG} so that the last identity in \eqref{so-comp-id} holds.

\begin{example}\label{Y1-SO}
	Building on Example~\ref{ex:phs}, an $I$-graded $\mathcal{SO}$-complex for $Y_1 = \Sigma(2,3,5)$ is given as follows. 
	(See, for example, \cite[Example 3.21]{Da:CS-HCG} for more details.)  
	The $\Q[x^{\pm 1}]$-module $C(Y_1)$ has rank $2$ with generators $\alpha$ and $\beta$ whose $(\Z\times \R)$-gradings are
	given in \eqref{bi-grading}. The differential $\widetilde d$ on $\widetilde C(Y_1)$ has trivial components 
	$d$ and $D_2$. The map $D_1$ sends $\alpha$ to $1$, and the map $U$ sends $\beta$ to $4\alpha$, and $\alpha$ to $6x^{-1}\beta$.
\end{example}
\begin{remark}
	For an arbitrary homology sphere, in order to define instanton Floer homology, one needs to apply a perturbation to the Chern--Simons 
	functional to obtain the Morse-Smale condition.  One can construct an $I$-graded $\SOt$-complex for any appropriate perturbation.  
	Here the $I$-grading is defined using the value of the perturbed Chern--Simons functional. 
	The $I$-graded $\SOt$-complexes for two different perturbations need not be homotopy equivalent $I$-graded 
	$\SOt$-complexes.  This necessitates the definition of the notion of enriched $\SOt$-complexes, defined in a similar way to enriched 
	$\mathcal S$-complexes in \cite[Section 7]{DS:equiv-asp-sing}, to generalize the above construction 
	for general homology spheres.  This basically consists of packaging together an infinite sequence of perturbed instanton $\SOt$-complexes as 
	the perturbations converge to zero. One key point is that any $I$-graded $\SOt$-complex is canonically an enriched $\SOt$-complex.
\end{remark}

Recall that our goal is to understand the instanton invariants of $Y_n$.  Note that $Y_n$ does not satisfy the non-degeneracy condition for the Chern--Simons functional, since the critical sets have positive dimension.  In principle, we would need an enriched $\SOt$-complex to study the instanton invariants.  However, we can circumvent this issue with the following connected sum theorem.  

\begin{theorem}\label{thm:kunneth}
	Suppose $Y$ and $Y'$ are two integer homology spheres whose enriched instanton  $\mathcal{SO}$-complexes can be represented by $I$-graded $\mathcal{SO}$-complexes 
	$(\widetilde C(Y),\widetilde d,\chi)$ and $(\widetilde C'(Y),\widetilde d',\chi')$. Then the instanton enriched $\mathcal{SO}$-complex of $Y\mathbin{\#}Y'$ can be represented by the $I$-graded $\mathcal{SO}$-complex given by the tensor product of 
	$\widetilde C(Y)$ and $\widetilde C'(Y)$. 
\end{theorem}
\begin{proof}
	The essential ingredients for the proof of this connected sum theorem are given in \cite{FukayaConnectSum} and \cite[Section 7.4]{Don:YM-Floer}, 
	and the proof can be completed by following the scheme in \cite[Section 6]{DS:equiv-asp-sing}.
\end{proof} 

The above theorem now allows us to study $Y_n$ with an $I$-graded $\SOt$-complex, which we can compute when combined with Example~\ref{Y1-SO}.  In the next section, we will define the numerical invariant $\Gamma$ of $I$-graded $\SOt$-complexes, and carry out the computation of the $\Gamma$-invariant of a connected sum of Poincar\'e homology spheres.  Combined with Theorem~\ref{thm:big-extension}, proved in the final section, we will be able to quickly prove Theorem~\ref{thm:extension}.

\begin{remark}
	An alternate route to proving the main theorem of this paper can be given using the knot invariants of \cite{DS:equiv-asp-sing} applied to the unknot 
	in the connected sum $Y_n$ of Poinacr\'e homology spheres. To be a bit more specific, \cite{DS:equiv-asp-sing} introduces the notions
	of $I$-graded $\mathcal S$-complexes and enriched $\mathcal S$-complexes for knots in arbitrary integer homology spheres.
	A connected sum theorem for such objects similar to Theorem \ref{thm:kunneth} is established in \cite[Section 6]{DS:equiv-asp-sing}.
	Associated to the enriched $\mathcal S$-complex of a knot in an integer homology sphere, one can still define a numerical invariant $\Gamma$.
	Now the rest of the proof in this paper goes through without any essential change after replacing the invariant $\Gamma$ of $Y_n$ with 
	the $\Gamma$ invariant of the unknot in $Y_n$. 
	We also comment that for any homology sphere $Y$, the $\Gamma$ invariant of the unknot in $Y$
	and the $\Gamma$ invariant of $Y$ are expected to have equivalent information (see \cite[Subsection 5.3]{DS:clasp}).
\end{remark}


\section{The $\Gamma$-invariant for a connected sum of Poincar\'e homology spheres}\label{sec:gamma-phs}

The $\Gamma$-invariant is a numerical invariant $\Gamma_{\widetilde{C}}\colon \Z \to \overline{\R}_{\geq 0}$ defined for any $I$-graded $\SOt$-complex (and more generally enriched $\SOt$-complexes).  Here, $\overline{\R}_{\geq 0}$ denotes $\R_{\geq 0} \cup \{ \infty \}$.  Since we will be only concerned with $I$-graded $\SOt$-complexes in the present paper, we recall the definition only for $I$-graded $\SOt$-complexes.  

Consider an $I$-graded $\mathcal{SO}$-complex $(\widetilde C,\widetilde d,\chi)$ wth the associated maps $d$, $U$, $D_1$ and $D_2$.  We define $\Gamma_{\widetilde C}(i)$ for a positive integer $i$ as 
\begin{equation}\label{Gamma-positive}
	\Gamma_{\widetilde C}(i):=\min_\alpha\left\{\deg_I(\alpha) \;\; \Big|\;\;  \begin{array}{c}
	 \alpha\in C_{4i-3},\, \, D_1 U^{j}\alpha=0\,\, \text{ for }\,\, 0\leq j\leq i-2,\\[2mm] D_1 U^{i-1}\alpha \neq 0,\,d\alpha=0\end{array}\right\}
\end{equation}
and for a non-positive integer $i$ as
\begin{equation}\label{Gamma-negative}
	\Gamma_{\widetilde C}(i):=\min_\alpha\left\{\max(\deg_I(\alpha),0)\;\; \Big|\;\; 
	  \begin{array}{c}\alpha\in C_{4i-3},\, \, a_j\in R[U^{\pm 1}]\,\, (0\leq j\leq -i),\, a_{-i}=1,\, \\[2mm]
	 d\alpha=\sum_{j=0}^{-i}U^{j}D_2(a_j) \end{array}\right\},
\end{equation}
where $C_k$ denotes the subspace of $C$ given by elements $\alpha$ with ${\gr}(\alpha)=k$. If the set used in \eqref{Gamma-positive} or \eqref{Gamma-negative} is empty, then $\Gamma_{\widetilde C}(i)=\infty$. Note that our assumptions imply that $\Gamma_{\widetilde C}(i)$ is non-negative, and if $i$ is small enough, then $\Gamma_{\widetilde C}(i)$ vanishes. 

We can now define the $\Gamma$-invariant for a restricted class of integer homology spheres.  
\begin{definition}
If the instanton Floer complex for $Y$ can be represented by an $I$-graded $\SOt$-complex $\widetilde{C}(Y)$, then we define $\Gamma_Y$ to be $\Gamma_{\widetilde{C}(Y)}$.  
\end{definition}

\begin{example}
Let $Y_1 = \Sigma(2,3,5)$.  Then it follows from Example~\ref{Y1-SO} that   
\[
\Gamma_{Y_1}(i)= \left\{
\begin{array} {ll}
0 & i \leq 0, \vspace{0.2cm}\\
\frac{1}{120} & i = 1, \vspace{0.2cm}\\
\frac{49}{120} & i = 2, \vspace{0.2cm}\\
\infty & i > 2.
\end{array}\right.
\]
\end{example}

The rest of this section is devoted to computing $\Gamma$ for a connected sum of Poincar\'e homology spheres.  
\begin{prop}\label{prop:Gamma-Yn}
	The values of $\Gamma_{Y_n}$ are given by
	\begin{equation}\label{Gammani}
		\Gamma_{Y_n}(i)=\left\{
		\begin{array}{ll}
			0&i\leq 0,\vspace{0.2cm}\\
			\frac{i}{120}& 0<i\leq n,\vspace{0.2cm}\\
			\frac{n}{120}+(i-n)\frac{2}{5}& n< i \leq 2n,\vspace{0.2cm}\\
			\infty&i>2n.
		\end{array}
		\right.
	\end{equation}
\end{prop}

\begin{proof}
We combine Theorem~\ref{thm:kunneth} with our computation of Example \ref{Y1-SO} to produce an $I$-graded $\mathcal{SO}$-complex $(\widetilde \fC(n),\widetilde \fd_n,\chi_n)$ for $Y_n$. There is a splitting of $\widetilde \fC(n)$ as $\fC(n)\oplus \fC(n)\oplus \Q[x^{\pm 1}]$ such that $\fC(n)$ and 
	the associated maps $\fd_n$, $\fU_n$, $\fD_{1,n}$ and $\fD_{2,n}$, defined as in \eqref{widetilde-d}, can be described inductively in the following way.
	
	For $n=1$, the module $\fC(1)$ is given in Example \ref{Y1-SO}, and for $n\geq 2$, we have
	\begin{equation} \label{chain-cx}
		\fC(n)=\fC(1)\otimes \fC(n-1)\oplus \underline \fC(1)\otimes \fC(n-1)\oplus \fC(1)\oplus \fC(n-1) ,
	\end{equation}
	where $\underline \fC(1)$ is $\fC(1)$ with $\gr$ being shifted up by $3$. 
	In particular, we may write a basis for $\fC(n)$ given by elements of the form
	\begin{equation} \label{gen}
	  \zeta=\sigma_1\otimes \sigma_2\otimes \dots \otimes \sigma_n
	\end{equation}
	where $\sigma_j\in \{\Theta,\alpha,\beta,\underline \alpha,\underline \beta\}$ such that for the largest $j$ that $\sigma_j\neq \Theta$, we have $\sigma_j\in \{\alpha,\beta\}$. Here $\Theta$ represents $1\in \Q$, and  $\underline \alpha, \underline \beta\in \underline \fC(1)$ 
	denote the elements corresponding to $\alpha$ and $\beta$. In particular, $\underline \alpha$ and $\underline \beta$ respectively have degrees $4$ and $8$.
	
	The maps $\fd_n$ and $\fU_n$ have the following inductive description with respect to the splitting in \eqref{chain-cx}.  (See \cite[Section 4.5]{DS:equiv-asp-sing}.)
	\begin{equation} \label{diff}
	\fd_{n}=\left[
	\begin{array}{cccc}
		-1\otimes \fd_{n-1}&0&0&0\\
		\fU_1\otimes 1-1\otimes \fU_{n-1}&1\otimes \fd_{n-1}&0&0\\
		-1\otimes \fD_{1,n-1}&0&0&0\\
		\fD_{1,1}\otimes 1&0&0&\fd_{n-1}
	\end{array}
	\right],
	\end{equation}
	\begin{equation}\label{U}
	\fU_{n}=\left[
	\begin{array}{cccc}
		\fU_1\otimes 1&0&0&0\\
		0&\fU_1\otimes 1&0&0\\
		0&0&\fU_1&0\\
		0&\fD_{1,1}\otimes 1&0&\fU_{n-1}
	\end{array}
	\right].
	\end{equation}
	Moreover, we have
	\begin{equation} \label{D1}
	\fD_{1,n}=\left[
	\begin{array}{cccc}
		0&0&\fD_{1,1}&\fD_{1,n-1}
	\end{array}
	\right],\hspace{1cm}
	\fD_{2,n}=0\
	\end{equation}
	We make the following observations about these maps and $\zeta$ as in \eqref{gen}:
	\begin{itemize}
	\item[(i)] $\fD_{1,n}(\zeta)$ is non-trivial if and only if $\sigma_j=\Theta$ for all $j$ except one in which case $\sigma_j$ is equal to $\alpha$.
	\item[(ii)] If there are at least two $\sigma_j$ in the expression \eqref{gen} which belong to the set $\{\alpha,\beta\}$, then $\fD_{1,n}\fU_n^l(\zeta)$, 
		for any non-negative integer $l$ is equal to $0$. 
	\end{itemize}
	Motivated by (ii), we say $\zeta$ in \eqref{gen} is {\it special} if there is at most one (and hence exactly one) $\sigma_j$ in $\{\alpha,\beta\}$.

	Let $\zeta$ in \eqref{gen} be special. The description of $\fd_n$ shows that $\zeta$ is a cycle. Suppose $a$ and $b$ respectively denote the number of $j$ that $\sigma_j\in \{\alpha,\underline \alpha\}$ and $\sigma_j\in \{\beta,\underline \beta\}$.
	Then the Floer grading of $\zeta$ is $4a+8b-3$.  In particular, for a positive integer $i$, $\fD_{1,n}\fU_n^{i-1}(\zeta)$ is non-zero only if $a$ and $i$ have the same parity. If we let
	\[
	  m=\frac{i-a}{2}-b,
	\]
	then the renormalized special cycle $x^{m}\zeta$ has Floer grading $4i-3$, and 
	\begin{equation}\label{deg-I-zeta}
	  \deg_I(x^{m}\zeta)=\frac{i}{2}-\frac{59}{120}a-\frac{71}{120}b.
	\end{equation}
	Using \eqref{U} and \eqref{D1}, it is easy to check that $\fD_{1,n}\fU_n^{i-1}(x^m\zeta)$, for a special $\zeta$ as in \eqref{gen} has the form
	\begin{equation}
		x^{m}\sum_{\mathcal F=(i_1,i_2,\dots,i_n)\in \Z_{\geq 0}^n}F_{i_1}(\sigma_1) \cdot F_{i_2}(\sigma_2) \dots F_{i_n}(\sigma_n),
	\end{equation}
	where the sum is over all tuples of non-negative integers $\mathcal F=(i_1,i_2,\dots,i_n)$ such that $i_1+\dots+i_n=i$, $i_j=0$ if $\sigma_j=\Theta$, $i_j$ is odd if $\sigma_j\in \{\alpha,\underline \alpha\}$ and $i_j$ is a positive even integer if 
	$\sigma_j\in \{\beta,\underline \beta\}$.
	For a positive integer $k$, $F_{k}(\sigma)=\fD_{1,1}\fU_1^{k-1}(\sigma)$ and $F_{0}(\sigma)=1$.


	To compute $\Gamma_{Y_n}(i)$, we wish to minimize \eqref{deg-I-zeta} among all special cycles $\zeta$ as above where $a$ and $i$ have the same parity, and $\fD_{1,n}\fU_n^{i-1}(x^m\zeta)$ is non-trivial. A necessary condition for the latter 
	property is the existence of $\mathcal F$ satisfying the properties mentioned in the previous paragraph. If $1\leq i\leq n$, then the minimum value of \eqref{deg-I-zeta} among all special cycles $\zeta$, for which there is ${\mathcal F}$ satisfying 
	required properties, can be achieved, for example, by 
	\begin{equation}\label{special-min-case-1}
		\zeta=\underbrace{ \Theta\otimes \dots \Theta}_{n-i}  \otimes \underbrace{\underline \alpha \otimes \dots \otimes \underline \alpha}_{i-1}\otimes \alpha,
	\end{equation}	
	and the only $\mathcal F$ for this special element is 
	\begin{equation*}
		{\mathcal F}=(\underbrace{ 0, \dots,0}_{n-i}, \underbrace{1,\dots, 1}_{i}),
	\end{equation*}	
	which shows that $\fD_{1,n}\fU_n^{i-1}(\zeta)$ is non-zero. On the other hand, there is no $\mathcal F$ for $\zeta$ in \eqref{special-min-case-1} if we replace $i$ with a smaller integer, which shows that $\fD_{1,n}\fU_n^{j}(\zeta)=0$ if $j<i-1$.
	Thus $\Gamma_{Y_n}(i)$ is equal to the value of \eqref{deg-I-zeta} for $\zeta$ in \eqref{special-min-case-1}, which is equal to $\frac{i}{120}$.

	In the case that $n+1\leq i\leq 2n$, the minimum value of \eqref{deg-I-zeta} among all special elements $\zeta$, with a tuple ${\mathcal F}$ satisfying the required properties, can be achieved, for example, by 
	\begin{equation}\label{special-min-case-2}
		\zeta=\underbrace{\underline \alpha \otimes \dots \otimes \underline \alpha}_{2n-i}\otimes \underbrace{ \underline \beta\otimes \dots \underline \beta}_{i-n-1}  \otimes \beta,
	\end{equation}	
	and the only admissible ${\mathcal F}$ is 
	\begin{equation*}
		{\mathcal F}=(\underbrace{1,\dots, 1}_{2n-i},\underbrace{2,\dots, 2}_{i-n}).
	\end{equation*}
	Analogous to the previous case, we see that for $\zeta$ in \eqref{special-min-case-2} the smallest $j$ that $\fD_{1,n}\fU_n^{j}(\zeta)\neq 0$ is $i-1$. Therefore, $\Gamma_{Y_n}(i)=\frac{n}{120}+(i-n)\frac{2}{5}$, which is obtained 
	by evaluating \eqref{deg-I-zeta} for $\zeta$ in \eqref{special-min-case-2}. For a non-positive integer $i$, vanishing of $\fD_{2,n}$ shows that $\Gamma_{Y_n}(i)=0$. It is shown in \cite{Da:CS-HCG} that $\Gamma_Y(i)=\infty$ for any $i>2h(Y)$, where $h$ denotes 
	the Fr\o yshov's invariant introduced in \cite{Fro:h-inv}. In particular, Fr\o yshov shows $h(Y_n)=n$, and we see $\Gamma_{Y_n}(i)=\infty$ for $i>2n$.
\end{proof}

\section{Extension of representations to cobordisms}\label{sec:extensions}
\subsection{ASD moduli spaces}
In order to prove Theorems~\ref{thm:extension} and~\ref{thm:big-extension}, we need to review some gauge-theoretic constructions on four-manifolds.  Let $W$ be a cobordism from $Y$ to $Y'$ equipped with a Riemannian metric and attach cylindrical ends, i.e.\ isometric copies of $(-\infty, 0] \times Y$ and $[1,\infty) \times Y'$.  We will not distinguish between this 4-manifold with cylindrical ends and the original compact cobordsim, where it does not cause confusion.  For a connection $A$ on the trivial $SU(2)$ bundle on $W$, we say $A$ is an {\em anti-self-dual connection} if $F_A^+ = 0$.  If $W$ is isometric to $Y \times \R$, then this is the same as asking that $A$ is gauge equivalent to a downward gradient flowline of the Chern-Simons functional.  Define the topological energy of a connection $A$ to be
\[
\energy(A) = \frac{1}{8\pi^2} \int_{W} {\rm tr}(F_A \wedge F_A).
\]
Note that if $A$ is ASD, then $\energy(A)  = 8 \pi^2\| F_A \|_{L^2} \geq 0$ and $A$ is flat if and only if $\energy(A) = 0$.  If $A$ is an ASD connection with finite energy, then it is asymptotic to flat connections on $Y$ at $-\infty$ and $Y'$ at $+\infty$.  For $\alpha, \alpha'$ critical points of $\CS$ on $Y, Y'$, we write $M(W,\alpha, \alpha')_d$ to denote the component of the moduli space of gauge equivalence classes of ASD connections which are asymptotic to $\alpha$, $\alpha'$ on the ends and has {\it expected dimension $d$}.  
Further, 
\[
\CS(\alpha) - \CS(\alpha') \equiv  \energy(A)\mod \Z.
\]  
In particular, if we fix a lift $\widetilde \alpha$ of $\alpha$, then the connection $A$ as above allows us to fix a lift $\widetilde \alpha'$ of $\alpha'$ by requiring that $\CS(\widetilde \alpha)$ is equal to $\CS(\widetilde \alpha')+\energy(A)$.

Since the Chern-Simons functional on $Y_n$ has degenerate critical points, we will ultimately need to introduce some perturbations in three- and four-dimensions.  We explain this briefly.  To fix degeneracies of the Chern-Simons functional, one applies a small perturbation, called a {\em holonomy perturbation}.  Since we will not need the precise definition, we only point out that these come from functions $\pi \colon \A(Y) \to \R$ that are governed by the traces of holonomies around a collection of loops; holonomy perturbations are gauge invariant.  For a generic holonomy perturbation $\pi$, the critical points of $\CS + \pi$  become non-degenerate where non-degeneracy is now defined using non-degeneracy of the perturbation of $K_\alpha$ by the negative of the Hessian of $\pi$.  These critical points are no longer in correspondence with representations of the fundamental group.  Since this is a generic condition, we may always choose our holonomy perturbations to be arbitrarily small in a suitable sense.  Note that $\CS + \pi$ might have more or fewer critical points than $\CS$ (although there are a finite number on $\B$). The critical {\em values} of $\CS + \pi$ may not be the same as those of $\CS$; however, as $\pi$ approaches zero, the critical values of $\CS + \pi$ converge to those of $\CS$.  

Now, suppose we have $W$ as above.  The perturbations of Chern-Simons by $\pi, \pi'$ on the ends can be used to induce a perturbation $\Pi$ on $W$ which vanishes away from the cylindrical ends; however, the corresponding moduli space of perturbed ASD connections might not be cut out transversely. This can be fixed by altering the perturbation $\Pi$ of the ASD equation on $W$ in a compact set by a term which can be made arbitrarily small.  In particular, if $\pi$ and $\pi'$ are small, we may also choose $\Pi$ to be small too.  We call such a perturbation $\Pi$ where the moduli spaces of ASD connections are cut out transversely an {\em admissible perturbation}.  
  

\subsection{Extension theorems from $\Gamma$}
One advantage of $\Gamma$ is that when it is non-trivial, we can use it to obtain information about flat connections over four-manifolds.  While we have only discussed the definition of $\Gamma$ for certain homology spheres, we will rely on the following lemma, derived from \cite[Proof of Theorem 3.45]{Da:CS-HCG}, which applies to arbitrary homology spheres.

\begin{lemma}\label{lem:extension}
	Suppose $W : Y \to Y$ is a negative definite four-manifold with $b_1 = 0$ and $0 < \Gamma_Y(i) < \infty$.  For any $\epsilon > 0$ there is $\delta>0$ such that the following holds. 
	Suppose $\pi$ is a perturbation of the Chern--Simons functional on $Y$ whose norm is smaller than $\delta$ and the critical points of $\CS+\pi$ are non-degenerate.
	Then there is a perturbation $\Pi$ of the ASD equation on $W$ with norm at most $\epsilon$ and compatible with the perturbation $\pi$ of the Chern--Simons functional of $Y$ on the ends 
	and a $\Pi$-perturbed
	 ASD connection $A$ on $W$
	such that $\| F(A) \|_{L^2} \leq \epsilon$, the connection $A$ is asymptotic to $\pi$-perturbed lifted flat connections $\alpha$, $\alpha'$ on $Y$, 
	${\gr}(\alpha) = {\gr}(\alpha') = 4i-3$ and $(\CS + \pi)(\alpha)$ and 
	$(\CS + \pi)(\alpha')$ are $\epsilon$-close to $\Gamma_Y(i)$.  
\end{lemma}

In the setting that the Chern-Simons functional is Morse--Bott, we are able to use the extension lemma above to find full components of the character variety for which every representation extends over any negative definite cobordism, proving Theorem~\ref{thm:big-extension}.  

\begin{theorem}\label{thm:big-extension-2}
	Suppose the Chern--Simons functional on an integer homology sphere $Y$ is Morse--Bott and $0<\Gamma_Y(i)<\infty$. Let $W\colon Y \to Y$ be a definite four-manifold. 
	Then there is a connected component $\chi_0$ of the character variety of $Y$ such that 
	\begin{itemize}
		\item[(i)] $\CS(\chi_0)=\Gamma_Y(i)$,
		\item[(ii)] any element of $\chi_0$ extends over $W$,
		\item[(iii)] ${\gr}(\chi_0)+\dim(\chi_0) \geq 4i-3$.
	\end{itemize}
\end{theorem}

Here we may use the same convention as in the discussion following the statement of Theorem \ref{thm:big-extension} in the introduction to make sense $\CS(\chi_0)$, ${\gr}(\chi_0)$ as numbers in $\R$, $\Z$. Equivalently, this can be interpreted in the following way that there is a lift of $\chi_0$ to $\mathcal B_0(Y)$ such that $\CS$ and ${\gr}$ of the lift of $\chi_0$ satisfy the conditions in the statement of Theorem \ref{thm:big-extension-2}. In fact, any real valued lift of $\CS(\chi_0)$ determines uniquely a lift of any given connected component $\chi_0$.

\begin{proof}
	Let $f_{\chi}$ be a Morse function on $\chi = \chi(Y)$.  It is shown in \cite[Theorem 5.12]{SavelievHomology} that there is a holonomy perturbation $f\colon\mathcal B(Y) \to \R$ such 
	that $f|_{\chi(Y)} = f_\chi$, and for any $0<\epsilon<1$, all the critical points of $\CS+\epsilon f$ are non-degenerate. 
	After multiplying $f$ by a small factor if necessary, we have good control over the critical points of $\CS+\epsilon f$ for any $0<\epsilon<1$.  
	In particular, the critical points of $\CS + \epsilon f$ are contained in a small neighborhood of $\chi$ and can be identified with the critical points of $f_{\chi}$. Let $\fC(f)$ denote 
	the set of critical points of $f_{\chi}$. For any $\alpha\in \fC(f_\chi)$, there is a critical point $\alpha_\epsilon$ of  $\CS+\epsilon f$ such that $\alpha_\epsilon \to \alpha$ as $\epsilon\to 0$. Moreover, if $\alpha$ belongs to the connected component $\chi_0$ of $\chi$,
	then we have 
	\[
	  {\rm gr}(\alpha_\epsilon)={\gr}(\chi_0)+\mu(\alpha;f_{\chi}),\hspace{1cm}\lim_{\epsilon\to 0}\CS(\alpha_\epsilon)=\CS(\chi_0)
	\]
	where ${\rm gr}(\alpha_\epsilon)$ is the Floer grading of $\alpha_\epsilon$ defined using the perturbation of $K_\alpha$ by the negative of the gradient of $\epsilon f_\chi$ and $\mu(\alpha;f_{\chi})$ is the Morse index of $f_{\chi}$ at the critical point $\alpha$. 
	(See, for example, \cite[Lemma 7]{ConnectSumSU3}.)
	
	Suppose $W\colon Y\to Y$ is given as in the statement.
	By changing the orientation of $W$ and flipping it upside down if necessary, we may assume that $W$ is negative definite. 
	By surgering out a set of loops giving a basis for $b_1(W)$, we obtain a negative definite cobordism with trivial $b_1$ 
	whose fundamental group is a quotient of the fundamental group of $W$.
	Therefore, we only need to prove the claim for this new cobordism and hence we can assume that $b_1(W)=0$.
	Since $0<\Gamma_Y(i)<\infty$, 
	by applying Lemma~\ref{lem:extension} to an appropriate sequence $\delta_k f$ of perturbations of the Chern--Simons functional with $\delta_k \to 0$, 
	we obtain $(\delta_k f)$-perturbed flat connections $\alpha_k$ and 
	$\beta_k$ with ${\rm gr}(\alpha_k)={\rm gr}(\beta_k)=4i-3$ and a $\Pi_k$-perturbed ASD connection  $A_k$ over $W$ such that 
	$\Pi_k\to 0$ and $|\!|F_{A_k}|\!|\to 0$ as $k\to \infty$.  
	We also know that the values of the perturbed Chern-Simons functional on $\alpha_k$ and $\beta_k$ are close to $\Gamma_Y(i)$.  

	Suppose that $\alpha_k$ and $\beta_k$ correspond to elements of the components $\chi_0$ and $\chi'_0$ of $\chi$.  After possibly passing to a subsequence, $\alpha_k$ and $\beta_k$ converge to elements $\alpha$ and $\beta$ in the lifts $\widetilde \chi_0$ and $\widetilde \chi'_0$ of $\chi_0$, $\chi'_0$ where the choices of lifts are determined by the property that their Chern--Simons values are $\Gamma_Y(i)$.
	Moreover, $\alpha$ and $\beta$ are critical points of $f_\chi$.  Floer-Uhlenbeck compactness \cite{Uhlenbeck,Fl:I,Don:YM-Floer} implies that, after possibly passing to a subsequence, the gauge equivalence classes of the connections $A_k$ are chain convergent to
	$$\left([B_1],[B_2],\dots,[B_\ell],[A_\infty],[C_1],[C_2]\dots[C_m]\right)$$ where the $B_i$ and $C_j$ are ASD connections on $\R\times Y$ with positive energy, and $A_\infty$ is an ASD connection on $W$.  Since $\lim_{k \to \infty} \mathcal{E}(A_k) = 0$, we see that no energy can escape at the ends, and so $\ell = m = 0$.  Further, we see that $A_\infty$ is flat.  It follows that $A_\infty$ is asymptotic to $\alpha$ and $\beta$ on the incoming and the outgoing ends of $W$.  Furthermore, 
	\[
	  {\gr}(\widetilde \chi_0)+\mu(\alpha;f_{\chi}) = {\gr}(\widetilde \chi_0')+\mu(\beta;f_{\chi}) = 4i-3,
	\]
	and 
	\[
	  \CS(\widetilde \chi_0)=\CS(\widetilde \chi_0')=\Gamma_Y(i).
	\]	
Therefore, we have shown that $\alpha \in \widetilde \chi_0$ extends over $W$ and ${\gr}(\widetilde \chi_0) + \dim(\chi_0) \geq 4i - 3$.  We just need to see that any other element of $\chi_0$ extends over $W$ as well.  The elements of $\chi_0$ that extend over $W$ form a closed subset of $\chi_0$, so it suffices to see that all but finitely many elements in a connected component $\chi_0$ as above extend over $W$.  Suppose instead there was a collection $X_0$ of infinitely many elements in each $\chi_0$ that did not extend.  Then, we can choose a Morse function on each $\chi_0$ such that the critical points are all contained in $X_0$.  (For example, choose an arbitrary Morse function on $\chi_0$ and then pullback by a diffeomorphism which moves the critical points to points in $X_0$.)  By following the above construction, we obtain a flat connection in $\chi_0$ which extends over $W$ and which is a critical point of the Morse function with critical points in $X_0$ for some $\chi_0$.  
This is a contradiction.   
\end{proof}

With Theorem~\ref{thm:big-extension-2} established, we are able to prove Theorem~\ref{thm:extension}, that for any definite cobordism $W\colon Y_n \to Y_n$, there is a $3n$-dimensional component of $R(Y_n)$ such that every element extends over $W$. For notation, recall from Example~\ref{ex:phs} that there are two irreducible elements $\alpha$, $\beta$ of $\chi(\Sigma(2,3,5))$. 
\begin{cor}\label{cor:beta-extends}
	Suppose $W$ is a definite cobordism from $Y_n$ to $Y_n$. Then any element of the character variety $\chi(\beta,\beta,\dots,\beta)$ extends over $W$.
\end{cor}
\begin{proof}
	By Proposition~\ref{prop:Gamma-Yn}, we know that $0 < \Gamma_{Y_n}(2n) = \frac{49n}{120} < \infty$.  The Chern--Simons functional of $Y_n$ is Morse--Bott. 
	Therefore, by Theorem~\ref{thm:big-extension-2}, we know that there is a lift $\widetilde \chi_0$ of a component $\chi_0$ of $\chi(Y_n)$ such that every element of 
	$\chi_0=\chi(\vec\sigma)$ 
	extends over $W$, 
	$\CS(\widetilde \chi_0) = \frac{49n}{120}$, and ${\gr}(\widetilde  \chi_0) + \dim(\chi_0) \geq 8n-3$. We claim that $\chi(\beta, \beta, \ldots, \beta)$ is the  only connected component that 
	can satisfy this property.  
	Suppose $\chi_0 = \chi(\vec\sigma)$ where $\sigma_i\in \{\alpha,\beta,\Theta\}$.  If we define
	\[
	  (m_i,r_i)=\left\{
	  \begin{array}{ll}
	  	\left(4,\frac{1}{120}\right)&\sigma_i=\alpha,\vspace{0.2cm}\\
		\left(8,\frac{49}{120}\right)&\sigma_i=\beta,\vspace{0.2cm}\\		
		(0,0)&\sigma_i=\Theta,\\				
	  \end{array}
	  \right.
	\]
	then as a consequence of Example \ref{connected-sum-inv}, there is an integer $l$ determining the lift $\widetilde \chi_0$ of $\chi_0$ such that 
	\[
	  \CS\left(\chi(\vec\sigma)\right)=l+\sum_{i}r_i, \hspace{1cm} {\gr}\left(\chi(\vec\sigma)\right)+\dim(\chi(\vec\sigma))+3=8 l+\sum_{i}m_i.
	\]
	Suppose that $\beta$ appears $k < n$ times in $\vec\sigma$, $\alpha$ appears $j$ times, and $\theta$ appears $n - j - k$ times.  Then, we can compute 
	\begin{align}
	\CS(\chi(\vec\sigma))&= \frac{1}{8}({\gr}(\chi(\vec\sigma))+\dim(\chi(\vec\sigma))+3)+\sum_{i} \left(r_i-\frac{m_i}{8}\right)\nonumber\\
	  &\geq n +\sum_{i} \left(r_i-\frac{m_i}{8}\right)\nonumber\\
	  &= n +k\left(\frac{49}{120}-1\right) + j\left(\frac{1}{120}-\frac{1}{2}\right)\nonumber\\
	  &>n + n\left(\frac{49}{120}-1\right) \nonumber \\
	  &= \frac{49n}{120},  \end{align}
	which is in contradiction with $\CS(\chi(\vec\sigma))=\frac{49n}{120}$. Thus $\chi(\vec\sigma) = \chi(\beta,\ldots, \beta)$.
\end{proof}

\section{Extension of representations for more general 3-manifolds}

The goal of this section is to prove Theorem \ref{thm-Gamma} and verify the technical result required to apply Theorem \ref{thm-Gamma} to the case of Whitehead double of torus knots. First we need the following proposition about the behavior of $\Gamma_Y$ under the connected sum operation.

\begin{prop}[\cite{DISST}]\label{conn-sum-ineq}
	For any pair of integer homology spheres $Y$, $Y'$ and any pair of integers $k$, $k'$, we have the following inequality:
	\[
	  \Gamma_{ Y\# Y'}(k+k') \leq \Gamma_{Y}(k) + \Gamma_{Y'} (k').
	\]
\end{prop}
\begin{proof}
	In the case of the knot invariant $\Gamma_{(Y,K)}$ of \cite{DS:equiv-asp-sing} a similar inequality is proved in \cite{DISST}. The key ingredient 
	in the proof of this inequality is the connected sum theorem proved in \cite[Section 6]{DS:equiv-asp-sing}. Similarly, combining the 
	connected sum theorem in Theorem \ref{thm:kunneth} and the algebraic arguments of \cite{DISST} gives the desired result. 
\end{proof}

\begin{proof}[Proof of Theorem \ref{thm-Gamma}]
Suppose $Y$ is given as in the statement of Theorem \ref{thm-Gamma}. The real valued quantity $\gr-8\CS$ is constant on each connected component of $\chi(Y)$, and hence it takes its maximum value on the compact set $\chi(Y)$, which we denote by $C_0$. More generally, the Chern--Simons functional of $\#_nY$ for any $n$ is Morse--Bott, and the critical set (and its gauge theoretical invariants) can be described following Example \ref{connected-sum-inv}. Furthermore, Proposition \ref{conn-sum-ineq} implies that there is a fixed positive constant $\epsilon$ independent of $n$ such that 
\begin{equation}\label{Gamma-bound}
	\Gamma_{ \#_nY}(i_0n)\leq \left(\frac{i_0}{2}-\epsilon\right)n
\end{equation}
for any positive integer $n$.  Now suppose that $W\colon \#_nY \to \#_nY$ is a definite cobordism. Theorem \ref{thm:big-extension} applied to the 3-manifold $\#_nY$ and the integer $i_0n$ asserts that there is a connected component $\chi(n)$ of $\chi(\#_nY)$ such that 
\begin{equation}\label{ineq-component}
  {\gr}(\chi(n))-8\CS(\chi(n))+ \dim(\chi(n))\geq 4i_0n-8\Gamma_Y(i_0n) - 3
\end{equation}
and any element of $\chi(n)$ regarded as an element of the character variety on the incoming end of $W$ extends to an element of the character variety of $W$. (Note that Theorem \ref{thm:big-extension} allows us to control the restriction of this extension on the outgoing end, too. However, we do not need this stronger result here.) Combining \eqref{Gamma-bound} and \eqref{ineq-component} shows that 
\begin{equation}\label{ineq-CS-G}
	{\gr}(\chi(n))-8\CS(\chi(n))\geq 8\epsilon n-3-\dim(\chi(n)).
\end{equation}

We claim that the dimension of $\chi(n)$ grows without any bound as $n$ goes to infinity. Otherwise, let $3N_0$ be an upper bound on the dimensions of $\chi(n)$. Then $\chi(n)$ is given by gluing irreducible flat connections on $k$ summands and the trivial connection on the remaining $n-k$ summands of $\#_nY$ with $k\leq N_0+1$. Using Example \ref{connected-sum-inv} and the definition of the constant $C_0$, we have
\[
  {\gr}(\chi(n))-8\CS(\chi(n))\leq kC_0,
\]
which is bounded because $k\leq N_0+1$. On the other hand, \eqref{ineq-CS-G} implies that ${\gr}(\chi(n))-8\CS(\chi(n))$ tends to infinity as $n$ goes to infinity. This contradiction implies that the dimension of $\chi(n)$ is unbounded as $n$ goes to infinity. Arguing in the same way as in the proof of Theorem \ref{thm:extension}, we can verify Theorem \ref{thm-Gamma}.
\end{proof}

Next, we verify that Theorem \ref{thm-Gamma} can be applied to the branched double cover of the Whitehead double $D(T_{p,q})$. As we mentioned in the introduction, we only need to check the Morse--Bott property.
\begin{prop}\label{Dpq-MB}
	For any $p,q$, the Chern--Simons functional of $Y_{p,q}$, 
	the branched double cover of the Whitehead double 
	of $D(T_{p,q})$ is Morse-Bott. 
\end{prop}

\begin{proof}
	We need to show that the character variety $\chi^*(Y_{p,q})$ is a smooth manifold and for any representation $\rho:\pi_1(Y_{p,q}) \to SU(2)$ representing an element of 
	$\chi^*(Y_{p,q})$, the dimension of the connected component of  $\chi^*(Y_{p,q})$ containing the class of $\rho$ agrees with $H^1(Y;{\rm ad}{\rho})$. Here  
	${\rm ad}{\rho}$ is the local coefficient system induced by the adjoint representation of $\rho$. The integer homology sphere $Y_{p,q}$ is given by $1/2$ surgery on the knot
	$T_{p,q}\#T_{p,q}$ \cite{MB:conc-dbl-br}. We use this description of $Y_{p,q}$ to obtain $\chi^*(Y_{p,q}))$ as the intersection of character varieties of two 3-manifolds with 
	torus boundary 
	over the character variety of a torus. 
	
	Any $SU(2)$-presentation of the torus $T^2$ is conjugate to a representation where the meridian $\mu$ and the longitude $\lambda$ of $T^2$ are respectively mapped to 
	$e^{2\pi i \alpha}$ and $e^{2\pi i\beta}$ for $(\alpha,\beta)\in \bR^2$. Translating $(\alpha,\beta)$ by an element of $\bZ^2$ or turning it into $(-\alpha,-\beta)$
	determines the same or conjugate representations. These actions of $\bZ^2$ and $\Z/2$ give an action of $G=\bZ^2\rtimes \bZ/2$ on 
	$\R^2$ and the quotient is the character variety of $T^2$,
	which is also called the {\it pillowcase}. We also write $\pi:\chi(T^2) \to [0,1/2]$ for the map induced by the projection $\R^2\to \R$ given by $(\alpha,\beta)\to \alpha$.

	For any knot $K$ in $S^3$, the inclusion map of the boundary torus into the knot exterior $Z(K)$ determines the {\it restriction} map
	$r:\chi(Z(K)) \to \chi(T^2)$; for any $\rho\in \chi(Z(K))$, we have $r(\rho)=(\alpha(\rho),\beta(\rho))$ where the trace of the holonomies of $\rho$ along the meridian 
	and the longitude of $K$ are respectively $\cos(2\pi\alpha(\rho))$ and $\cos(2\pi\beta(\rho))$. 
	For any $\alpha\in [0,1/2]$, there is a unique reducible representation $\theta_\alpha\in \chi(Z(K))$ such that $r(\theta_\alpha)$ is equal to the class of $(\alpha,0)$ in 
	$\chi(T^2)$. Any other representation in $\chi(Z(K))$ is irreducible and the set of all such representations is denoted by $\chi^*(Z(K))$.
	
	The character variety $\chi(Z(K))$ and the restriction map $r$ can be used 
	to determine $\chi^*(S^3_{1/n}(K))$. For any integer $n$, let $l_n$ be the subset of the pillowcase $\chi(T^2)$ determined by the line $\alpha+n\beta=0$ in $\bR^2$. 
	Then $\chi^*(S^3_{1/n}(K))$ is given by the intersection $r^{-1}(l_n)\cap \chi^*(Z(K))$. Equivalently, it is the fiber product of $\chi^*(Z(K)$ and $l_n$ over $\chi(T^2)$
	where the fiber product is taken with respect to the restriction map $r$ and the inclusion of $l_n$ in $\chi(T^2)$.
	
	In the case of $Y_{p,q}$, we can apply the above description to $K=T_{p,q}\#T_{p,q}$ and $n=2$ to characterize $\chi^*(Y_{p,q})$. The character variety $\chi(Z(T_{p,q}))$ 
	has a simple description given in \cite{Kl:SU-rep}. Using the Seifert--Van Kampen theorem, one can similarly characterize the character variety of $T_{p,q}\#T_{p,q}$
	together with the restriction map, and then check the 
	claim that the dimension of any connected component of $\chi^*(S^3_{1/2}(T_{p,q}\#T_{p,q}))$ containing a representation $\rho$ is equal to 
	the dimension of $H^1(S^3_{1/2}(T_{p,q}\#T_{p,q});{\rm ad}{\rho})$.
\end{proof}

\bibliography{references}

\begin{bibdiv}
\begin{biblist}

\bib{AGL}{article}{
      author={Aceto, Paolo},
      author={Golla, Marco},
      author={Lecuona, Ana~G.},
       title={Handle decompositions of rational homology balls and
  {C}asson-{G}ordon invariants},
        date={2018},
        ISSN={0002-9939},
     journal={Proc. Amer. Math. Soc.},
      volume={146},
      number={9},
       pages={4059\ndash 4072},
         url={https://doi-org.prox.lib.ncsu.edu/10.1090/proc/14035},
      review={\MR{3825859}},
}

\bib{Auckly}{article}{
      author={Auckly, David~R.},
       title={Topological methods to compute {C}hern-{S}imons invariants},
        date={1994},
        ISSN={0305-0041},
     journal={Math. Proc. Cambridge Philos. Soc.},
      volume={115},
      number={2},
       pages={229\ndash 251},
         url={https://doi-org.prox.lib.ncsu.edu/10.1017/S0305004100072066},
      review={\MR{1277058}},
}

\bib{Auckly:top-spec-flow}{article}{
      author={Auckly, David~R.},
       title={A topological method to compute spectral flow},
        date={1998},
        ISSN={1225-6951},
     journal={Kyungpook Math. J.},
      volume={38},
      number={1},
       pages={181\ndash 203},
      review={\MR{1637787}},
}

\bib{BCR}{book}{
      author={Bochnak, Jacek},
      author={Coste, Michel},
      author={Roy, Marie-Fran\c{c}oise},
       title={Real algebraic geometry},
      series={Ergebnisse der Mathematik und ihrer Grenzgebiete (3) [Results in
  Mathematics and Related Areas (3)]},
   publisher={Springer-Verlag, Berlin},
        date={1998},
      volume={36},
        ISBN={3-540-64663-9},
         url={https://doi-org.prox.lib.ncsu.edu/10.1007/978-3-662-03718-8},
        note={Translated from the 1987 French original, Revised by the
  authors},
      review={\MR{1659509}},
}

\bib{ConnectSumSU3}{article}{
      author={Boden, Hans~U.},
      author={Herald, Christopher~M.},
       title={A connected sum formula for the {${\rm SU}(3)$} {C}asson
  invariant},
        date={1999},
        ISSN={0022-040X},
     journal={J. Differential Geom.},
      volume={53},
      number={3},
       pages={443\ndash 464},
  url={http://projecteuclid.org.prox.lib.ncsu.edu/euclid.jdg/1214425635},
      review={\MR{1806067}},
}

\bib{Da:CS-HCG}{article}{
      author={Daemi, Aliakbar},
       title={Chern--{S}imons functional and the homology cobordism group},
        date={2020},
        ISSN={0012-7094},
     journal={Duke Math. J.},
      volume={169},
      number={15},
       pages={2827\ndash 2886},
         url={https://doi.org/10.1215/00127094-2020-0017},
      review={\MR{4158669}},
}

\bib{DISST}{misc}{
      author={Daemi, Aliakbar},
      author={Imori, Hayato},
      author={Sato, Kouki},
      author={Scaduto, Christopher},
      author={Taniguchi, Masaki},
       title={Instantons, special cycles, and knot concordance},
   publisher={arXiv},
        date={2022},
         url={https://arxiv.org/abs/2209.05400},
}

\bib{DLVVW}{article}{
      author={Daemi, Aliakbar},
      author={Lidman, Tye},
      author={Vela-Vick, David~Shea},
      author={Wong, C.-M.~Michael},
       title={Ribbon homology cobordisms},
        date={2022},
        ISSN={0001-8708},
     journal={Adv. Math.},
      volume={408},
       pages={Paper No. 108580},
         url={https://doi.org/10.1016/j.aim.2022.108580},
      review={\MR{4467148}},
}

\bib{Don:YM-Floer}{book}{
      author={Donaldson, Simon~K.},
       title={Floer homology groups in {Y}ang-{M}ills theory},
      series={Cambridge Tracts in Mathematics},
   publisher={Cambridge University Press, Cambridge},
        date={2002},
      volume={147},
        ISBN={0-521-80803-0},
         url={http://dx.doi.org/10.1017/CBO9780511543098},
        note={With the assistance of M. Furuta and D. Kotschick},
      review={\MR{1883043 (2002k:57078)}},
}

\bib{DS:equiv-asp-sing}{article}{
      author={{Daemi}, Aliakbar},
      author={{Scaduto}, Christopher},
       title={{Equivariant aspects of singular instanton Floer homology}},
        date={2019},
      eprint={arXiv:1912.08982},
}

\bib{DS:clasp}{article}{
      author={Daemi, Aliakbar},
      author={Scaduto, Christopher},
       title={Chern-{S}imons functional, singular instantons, and the
  four-dimensional clasp number},
        date={2024},
        ISSN={1435-9855,1435-9863},
     journal={J. Eur. Math. Soc. (JEMS)},
      volume={26},
      number={6},
       pages={2127\ndash 2190},
         url={https://doi.org/10.4171/jems/1320},
      review={\MR{4742808}},
}

\bib{Fl:I}{article}{
      author={Floer, Andreas},
       title={An instanton-invariant for {$3$}-manifolds},
        date={1988},
        ISSN={0010-3616},
     journal={Comm. Math. Phys.},
      volume={118},
      number={2},
       pages={215\ndash 240},
         url={http://projecteuclid.org/euclid.cmp/1104161987},
      review={\MR{956166}},
}

\bib{Fro:h-inv}{article}{
      author={Fr{\o}yshov, Kim~A.},
       title={Equivariant aspects of {Y}ang-{M}ills {F}loer theory},
        date={2002},
        ISSN={0040-9383},
     journal={Topology},
      volume={41},
      number={3},
       pages={525\ndash 552},
         url={http://dx.doi.org/10.1016/S0040-9383(01)00018-0},
      review={\MR{1910040}},
}

\bib{FS:pseudofree}{article}{
      author={Fintushel, Ronald},
      author={Stern, Ronald~J.},
       title={Pseudofree orbifolds},
        date={1985},
        ISSN={0003-486X},
     journal={Ann. of Math. (2)},
      volume={122},
      number={2},
       pages={335\ndash 364},
         url={https://doi.org/10.2307/1971306},
      review={\MR{808222}},
}

\bib{FS:HFSF}{article}{
      author={Fintushel, Ronald},
      author={Stern, Ronald~J.},
       title={Instanton homology of {S}eifert fibred homology three spheres},
        date={1990},
        ISSN={0024-6115},
     journal={Proc. London Math. Soc. (3)},
      volume={61},
      number={1},
       pages={109\ndash 137},
         url={https://doi.org/10.1112/plms/s3-61.1.109},
      review={\MR{1051101}},
}

\bib{FukayaConnectSum}{article}{
      author={Fukaya, Kenji},
       title={Floer homology of connected sum of homology {$3$}-spheres},
        date={1996},
        ISSN={0040-9383},
     journal={Topology},
      volume={35},
      number={1},
       pages={89\ndash 136},
         url={https://doi-org.prox.lib.ncsu.edu/10.1016/0040-9383(95)00009-7},
      review={\MR{1367277}},
}

\bib{GL}{article}{
      author={Gordon, Cameron~McA.},
      author={Luecke, John},
       title={Knots are determined by their complements},
        date={1989},
        ISSN={0894-0347},
     journal={J. Amer. Math. Soc.},
      volume={2},
      number={2},
       pages={371\ndash 415},
         url={https://doi-org.prox.lib.ncsu.edu/10.2307/1990979},
      review={\MR{965210}},
}

\bib{HKP}{article}{
      author={Hom, Jennifer},
      author={Kang, Sungkyung},
      author={Park, JungHwan},
       title={Ribbon knots, cabling, and handle decompositions},
        date={2021},
        ISSN={1073-2780},
     journal={Math. Res. Lett.},
      volume={28},
      number={5},
       pages={1441\ndash 1457},
         url={https://doi.org/10.4310/mrl.2021.v28.n5.a7},
      review={\MR{4471715}},
}

\bib{JMZ:tor}{article}{
      author={Juh\'{a}sz, Andr\'{a}s},
      author={Miller, Maggie},
      author={Zemke, Ian},
       title={Knot cobordisms, bridge index, and torsion in {F}loer homology},
        date={2020},
        ISSN={1753-8416},
     journal={J. Topol.},
      volume={13},
      number={4},
       pages={1701\ndash 1724},
         url={https://doi.org/10.1112/topo.12170},
      review={\MR{4186142}},
}

\bib{KirkKlassen}{article}{
      author={Kirk, Paul~A.},
      author={Klassen, Eric~P.},
       title={Chern-{S}imons invariants of {$3$}-manifolds and representation
  spaces of knot groups},
        date={1990},
        ISSN={0025-5831},
     journal={Math. Ann.},
      volume={287},
      number={2},
       pages={343\ndash 367},
         url={https://doi-org.prox.lib.ncsu.edu/10.1007/BF01446898},
      review={\MR{1054574}},
}

\bib{Kl:SU-rep}{article}{
      author={Klassen, Eric~Paul},
       title={Representations of knot groups in {${\rm SU}(2)$}},
        date={1991},
        ISSN={0002-9947,1088-6850},
     journal={Trans. Amer. Math. Soc.},
      volume={326},
      number={2},
       pages={795\ndash 828},
         url={https://doi.org/10.2307/2001784},
      review={\MR{1008696}},
}

\bib{Lickorish}{article}{
      author={Lickorish, W. B.~R.},
       title={A representation of orientable combinatorial {$3$}-manifolds},
        date={1962},
        ISSN={0003-486X},
     journal={Ann. of Math. (2)},
      volume={76},
       pages={531\ndash 540},
         url={https://doi-org.prox.lib.ncsu.edu/10.2307/1970373},
      review={\MR{151948}},
}

\bib{MB:conc-dbl-br}{article}{
      author={Manolescu, Ciprian},
      author={Owens, Brendan},
       title={A concordance invariant from the {F}loer homology of double
  branched covers},
        date={2007},
        ISSN={1073-7928,1687-0247},
     journal={Int. Math. Res. Not. IMRN},
      number={20},
       pages={Art. ID rnm077, 21},
         url={https://doi.org/10.1093/imrn/rnm077},
      review={\MR{2363303}},
}

\bib{Myers}{article}{
      author={Myers, Robert},
       title={Excellent {$1$}-manifolds in compact {$3$}-manifolds},
        date={1993},
        ISSN={0166-8641},
     journal={Topology Appl.},
      volume={49},
      number={2},
       pages={115\ndash 127},
         url={https://doi-org.prox.lib.ncsu.edu/10.1016/0166-8641(93)90038-F},
      review={\MR{1206219}},
}

\bib{SavelievHomology}{book}{
      author={Saveliev, Nikolai},
       title={Invariants for homology {$3$}-spheres},
      series={Encyclopaedia of Mathematical Sciences},
   publisher={Springer-Verlag, Berlin},
        date={2002},
      volume={140},
        ISBN={3-540-43796-7},
         url={https://doi-org.prox.lib.ncsu.edu/10.1007/978-3-662-04705-7},
        note={Low-Dimensional Topology, I},
      review={\MR{1941324}},
}

\bib{Taubes}{article}{
      author={Taubes, Clifford~H.},
       title={Gauge theory on asymptotically periodic {$4$}-manifolds},
        date={1987},
        ISSN={0022-040X},
     journal={J. Differential Geom.},
      volume={25},
      number={3},
       pages={363\ndash 430},
  url={http://projecteuclid.org.prox.lib.ncsu.edu/euclid.jdg/1214440981},
      review={\MR{882829}},
}

\bib{Uhlenbeck}{article}{
      author={Uhlenbeck, Karen~K.},
       title={Connections with {$L^{p}$} bounds on curvature},
        date={1982},
        ISSN={0010-3616},
     journal={Comm. Math. Phys.},
      volume={83},
      number={1},
       pages={31\ndash 42},
  url={http://projecteuclid.org.prox.lib.ncsu.edu/euclid.cmp/1103920743},
      review={\MR{648356}},
}

\bib{Wallace}{article}{
      author={Wallace, Andrew~H.},
       title={Modifications and cobounding manifolds},
        date={1960},
        ISSN={0008-414X},
     journal={Canadian J. Math.},
      volume={12},
       pages={503\ndash 528},
         url={https://doi-org.prox.lib.ncsu.edu/10.4153/CJM-1960-045-7},
      review={\MR{125588}},
}

\end{biblist}
\end{bibdiv}
\bibliographystyle{hplain}
\Addresses
\end{document}